\documentclass[12pt, reqno]{amsart}
\usepackage{amsfonts}
\usepackage{bbm}
\usepackage{amscd,amsfonts}
\usepackage{amssymb, eucal, amsfonts, amsmath, xypic, latexsym, tikz}
\usepackage{pifont}
\usepackage{mathrsfs,color}
\usepackage{amsthm,indentfirst,bm,fancyhdr,dsfont}
\usepackage{graphicx}
\usepackage[all]{xy}
\usepackage[CJKbookmarks=true]{hyperref}
\usepackage{extarrows}

\usepackage{mathrsfs}
\usepackage{amsmath}
\usepackage{amssymb}
\usepackage{hyperref}

\include{micro}

\newtheorem{thm}{Theorem}[section]
\newtheorem{lemma}[thm]{Lemma}
\newtheorem{remark}[thm]{Remark}
\newtheorem{corollary}[thm]{Corollary}

\theoremstyle{definition}
\newtheorem{defn}[thm]{Definition}
\newtheorem{proposition}[thm]{Proposition}

 \theoremstyle{remark}

\numberwithin{equation}{section}

\begin{document}

\def\frakl{{\mathfrak L}}
\def\frakg{{\mathfrak G}}
\def\bbf{{\mathbb F}}
\def\bbl{{\mathbb L}}
\def\bbz{{\mathbb Z}}
\def\bbr{{\mathbb R}}
\def\bbc{{\mathbb C}}

\def\bvp{\bf{\varphi}}

\def\ad{\textsf{ad}}
\def\GL{\text{GL}}
\def\Der{\mbox{\rm Der}}
\def\Hom{\textsf{Hom}}
\def\ind{\textsf{ind}}
\def\res{\textsf{res}}
\def\gl{\frak{gl}}
\def\sl{\frak{sl}}
\def\Ker{\textsf{Ker}}
\def\Lie{\textsf{Lie}}
\def\id{\textsf{id}}
\def\det{\textsf{det}}
\def\Lie{\textsf{Lie}}
\def\Aut{\textsf{Aut}}
\def\Ext{\textsf{Ext}}
\def\Coker{\textsf{{Coker}}}
\def\dim{\textsf{{dim}}}
\def\pr{\mbox{\sf pr}}
\def\SL{\text{SL}}

\def\geqs{\geqslant}

\def\ba{{\mathbf a}}
\def\bd{{\mathbf d}}
\def\bbk{{\mathbb K}}
\def\co{{\mathcal O}}
\def\cn{{\mathcal N}}
\def\cv{{\mathcal V}}
\def\cz{{\mathcal Z}}
\def\cq{{\mathcal Q}}
\def\cf{{\mathcal F}}
\def\cc{{\mathcal C}}
\def\ca{{\mathcal A}}

\def\sl{{\frak{sl}}}

\def\ggg{{\frak g}}
\def\lll{{\frak l}}
\def\hhh{{\frak h}}
\def\nnn{{\frak n}}
\def\sss{{\frak s}}
\def\bbb{{\frak b}}
\def\ccc{{\frak c}}
\def\ooo{{\mathfrak o}}
\def\ppp{{\mathfrak p}}
\def\uuu{{\mathfrak u}}

\def\p{{[p]}}
\def\modf{\text{{\bf mod}$^F$-}}
\def\modr{\text{{\bf mod}$^r$-}}

\title[Enhanced Tensor Invariants] {Enhanced Brauer algebras and enhanced dualities for orthogonal and symplectic groups}
\author{Bin Liu}
\address{School of Mathematical Sciences, East China Normal University, Shanghai, 200241, China.} \email{1918724868@qq.com}
\subjclass[]{}
 \keywords{Enhanced Tensor Invariant, Enhanced Brauer Algebra, Dualities}
\thanks{This work is partially supported by the NSF of China (Grant: 12071136).}

\begin{abstract} This note is a sequel to Shu-Xue-Yao's paper \cite{BYY} where the author studied the so-called enhanced groups and related dualities for type $A$. In this note, we continue to investigate the enhanced dualities for classical groups of type $B$, $C$, and $D$. To show this, we first introduce an enhanced Brauer algebra. Due to Proposition \ref{B dual}, we can easily obtain the restricted and Levi dualities, hence the parabolic duality is naturally described. Moreover, a special case is listed in \S\ref{3} by some explicit calculations.
\end{abstract}

\maketitle
\section*{0. Introduction}
The main purposes is to investigate the dualities related to Levi and parabolic groups in the case $G=\text{O}(V) \text{ or } \text{Sp}(V)$. When $G=\GL(V)$, the related dualities have already listed in \cite[Theorem 4.3, Theorem 5.3, Conjecture 5.4]{BYY} with the help of the  finite-dimensional degenerate double Hecke algebras $D(n,r)$. Similarly, an important algebraic model $\mathcal{B}(\varepsilon,n,r)$ is introduced which is called an enhanced Brauer algebras. The enhanced Brauer algebras will be very useful to establish the dualities mentioned above.

The main results are listed as follows.
\begin{thm}
	Let $G$ be $\text{O}(V)$ or $\text{Sp}(V)$ for a finite dimensional vector space $V$ over $\mathbb{C}$, and $\underline{V}=V\oplus C\eta$ an one-dimensional extension of $V$. Denote by $\underline{G}$ the enhanced group associated with $G$ and $\underline{V}$ in the sense of \cite{BYY}. The following statements hold.
	\item[(1)] (Theorem \ref{restricted}) Assume $B_{st}$ is defined as in \S\ref{RTIs}, we have 
	\begin{align*}
	\text{End}_G(\underline{V}^{\otimes r})=\left(\bigoplus_{s-t \neq 0 \text{ is even}}\mathcal{B}_{st}\right) \bigoplus \mathcal{B}(\varepsilon,n,r).
	\end{align*}	
	\item[(2)] (Theorem \ref{AB Dual}) $\text{End}_{G \times \textbf{G}_m}(\underline{V}^{\otimes r})=\mathcal{B}(\varepsilon,n,r).$
	\item[(3)] (Corollary \ref{inclusion3}) $\text{End}_{\underline{G} \rtimes \textbf{G}_m}(\underline{V}^{\otimes r})=\mathcal{B}(\varepsilon,n,r)^V.$
\end{thm}

\begin{thm} (Theorem \ref{inv thm})
	Assume $G=\text{Sp}(V)$ with $\dim V>2r$, or $G=\text{O}(V)$ with $\dim V \geq 2r$, then we have $\text{End}_{\underline{G} \rtimes \textbf{G}_m}(\underline{V}^{\otimes r})=\{\rho(f) \mid f \in \mathcal{B}_r(\varepsilon n)\}$. Here $\rho$ is defined in \S\ref{3} and \S\ref{algebra B}, and $\mathcal{B}_r(\varepsilon n)$ is defined in \S\ref{1.3}.
\end{thm}

We should note that the study of invariants beyond reductive groups is a challenge (see \cite{G1}, \cite{G2} and \cite{G3}). Consequently the invariant property of $\underline{G} \rtimes \textbf{G}_m$ as above has its own interest.

This note is organized as follows. We first introduce the enhanced Brauer algebras and give its basis ($n \geq 2r$) in \S\ref{1}. In \S\ref{2}, we first describe the restricted duality (see Theorem \ref{restricted}) with the help of the classical result Proposition \ref{B dual}. Next we give the Levi and parabolic dualities (see Theorem \ref{AB Dual} and \eqref{P Dual}). Finally, \S\ref{3} is devoted to the proof of Theorem \ref{inv thm} which is a special case of parabolic duality in \eqref{P Dual}.

\section{The enhanced Brauer algebras.}\label{1}

\subsection{A general setting-up.} In this note, all vector spaces are over a complex number field of $\mathbb{C}$. Let $G$ be a connected reductive algebraic group over $\mathbb{C}$, and $(V,\rho)$  be a finite-dimensional rational representation of $G$ with representation space $V$ over $\mathbb{C}$. We can define an enhanced reductive algebraic group $\underline{G}=G \times_{\rho} V$ as follows (see \cite{BYY} explicitly):

Regard $\underline{G}$ as a set, we have $\underline{G}=G \times V$; For any $(g_1, v_1), (g_2, v_2) \in G \times V$,
\begin{align}\label{def of group}
(g_1, v_1) \cdot (g_2, v_2):=(g_1g_2, \rho(g_1)(v_2)+v_1).
\end{align}

From now on, we will write down $e^v$ for $(1,v) \in \underline{G}$ ($1 \in G$ is the identity), and always suppose that $\dim V=n$, i.e. $V \cong \mathbb{C}^n$. Let $G_1$ and $G_2$ be subgroups of $G$, we denote by $\langle  G_1, G_2\rangle$ the subgroup of $G$ generated by $G_1$ and $G_2$.  All representations for algebraic groups are always assumed to be rational.

In particular, assume that $G$ is a closed subgroup of $\GL(V)$. Since $\GL(V)$ can naturally act on $V$, we can defined the enhanced reductive algebraic group $\underline{G}$ as above (where $\rho$ is just the natural representation of $G$ on $V$). Note that $V$ is a closed subgroup of $\underline{G}$, then an irreducible $G$-module becomes naturally an irreducible module of $\underline{G}$ with trivial $V$-action. The isomorphism classes of irreducible rational representations of $\underline{G}$ coincide with the ones of $G$. 

Let $\underline{V}:=V \oplus \mathbb{C}\eta \cong \mathbb{C}^{n+1}$, which is called the enhanced space of $V$. Naturally $\underline{V}$ becomes a $\underline{G}$-module that is defined for any $\underline{g}=(g,v) \in \underline{G}$ and $\underline{u}=u+a\eta \in \underline{V}$ with $g \in G$, $u,v \in V$ and $a \in \mathbb{C}$, via
\begin{align}\label{def of module}
\underline{g} \cdot \underline{u}:=\rho(g)(u)+av+a\eta.
\end{align}
It is easy to see that this module is a rational module of $\underline{G}$.

Since $\GL(V) \cong \GL_{n}$ and $\GL(\underline{V}) \cong \GL_{n+1}$, we can regard $\GL(V)$ and $\GL(\underline{V})$ as $\GL_{n}$ and $\GL_{n+1}$ respectively. Let $G$ be a closed subgroup of $\GL(V)$, there are canonical imbeddings of algebraic groups $G \hookrightarrow \GL_{n+1}$ given by 
\begin{align}
g \mapsto \text{diag}(g,1):=\left(\begin{matrix}
g & 0 \\
0 & 1
\end{matrix}\right), \forall g \in G.
\end{align}
and $\textbf{G}_m \hookrightarrow \GL_{n+1}$ given by \begin{align}
c \mapsto \text{diag}(1,\cdots,1,c):=\left(\begin{matrix}
I_n \\
 & c \\
\end{matrix}\right), \forall c \in \textbf{G}_m=\mathbb{C}^{\times}.
\end{align}
where $I_n$ is $n$ identity matrix and $\textbf{G}_m$ is the one-dimensional multiplicative algebraic group. In this sense, both $G$ and $\textbf{G}_m$ can be viewed as closed subgroups of $\GL_{n+1}$. 

Obviously, the closed subgroups $\langle G,\textbf{G}_m\rangle$ and $\langle \underline{G},\textbf{G}_m\rangle$ are isomorphic to $G \times {\textbf{G}}_m$ and $\underline{G} \rtimes {\textbf{G}}_m$ respectively. We identify all throughout the note.

Note that $\underline{G} \rtimes {\textbf{G}}_m$ is actually a parabolic subgroup of $\GL(\underline{V})$ associated with the Levi subgroup $G \times {\textbf{G}}_m$. Clearly, $\underline{G}$ is a closed subgroup of $\underline{G} \rtimes {\textbf{G}}_m$, then $e^v \in \underline{G} \rtimes {\textbf{G}}_m$ for all $v \in V$.

Obviously, if $G$ is a classical group, $G$ becomes a closed subgroup of $\GL(V)$. Now we suppose that $G$ is a closed subgroup of $\GL(V)$. The fact that $\GL(\underline{V})$ can naturally act on $\underline{V}$ implies that $\underline{V}$ can be viewed as a natural $\GL(\underline{V})$-module. Note that $\GL(V)$, $G$, $\underline{G}$, $G \times \textbf{G}_m$ and $\underline{G} \rtimes \textbf{G}_m$ are all closed subgroups of $\GL(\underline{V})$, hence $\underline{V}$ can be viewed as a natural $\GL({V})$-module, $G$-module, $\underline{G}$-module, $G \times \textbf{G}_m$-module and $\underline{G} \rtimes \textbf{G}_m$-module respectively. It is easy to see that the natural $\underline{G}$-module $\underline{V}$ coincides with (\ref{def of module}), where $\rho$ is just the natural representation of $G$ on $V$.

\subsection{Classical Schur-Weyl duality.}\label{1.2}
We can define a representation $(\underline{V}^{\otimes r},\Phi)$ of $\GL(\underline{V})$
\begin{equation}\label{tensor action}
\Phi(g)(v_1 \otimes v_2 \otimes \cdots \otimes v_r):=g(v_1) \otimes g(v_2) \otimes \cdots \otimes g(v_r).
\end{equation}
for any $g \in \GL(\underline{V})$ and any monomial tensor product $v_1 \otimes v_2 \otimes \cdots \otimes v_r \in \underline{V}^{\otimes r}$.

In the meanwhile, $(\underline{V}^{\otimes r},\Psi)$ naturally becomes a representation of $\frak S_r$ with following permutation action
\begin{equation}\label{left action}
\Psi(\sigma)(v_1 \otimes v_2 \otimes \cdots \otimes v_r):=v_{\sigma^{-1}(1)} \otimes v_{\sigma^{-1}(2)} \otimes \cdots \otimes v_{\sigma^{-1}(r)}.
\end{equation}
for any $\sigma \in \frak S_r$ and any monomial tensor product $v_1 \otimes v_2 \otimes \cdots \otimes v_r \in \underline{V}^{\otimes r}$.

The classical Schur-Weyl duality implies that the images of $\Phi$ and $\Psi$ are double centralizers in $\text{End}_{\mathbb{C}}(\underline{V}^{\otimes r})$, namely
\begin{equation}\label{CSWD}
\text{End}_{\GL(\underline{V})}(V^{\otimes r})=\mathbb{C}\Psi(\frak S_r), \text{    }
\text{End}_{\frak S_r}(\underline{V}^{\otimes r})=\mathbb{C}\Phi(\frak \GL(\underline{V})).
\end{equation}
Here $\mathbb{C}\Psi(\frak S_r)$ and $\mathbb{C}\Phi(\frak \GL(\underline{V}))$ are the subalgebras of $\text{End}_{\mathbb{C}}(\underline{V}^{\otimes r})$ that generated by $\Psi(\frak S_r)$ and $\Phi(\frak \GL(\underline{V}))$ respectively.

For any $m \in \mathbb{Z}_{>0}$, we define $\underline{m}:=\{1,2,\cdots,m\}$. Suppose $V=\bigoplus_{i=1}^n \mathbb{C}\eta_i$, then $\{\eta_1,\cdots,\eta_n,\eta\}$ is a basis of $\underline{V}$. Clearly, all $\eta_\textbf{i}:=\eta_{i_1} \otimes \cdots \otimes \eta_{i_r}$ for $\textbf{i}=\{i_1,\cdots,i_r\} \in (\underline{n+1})^r$ (i.e. $1 \leq i_1,\cdots i_r \leq n+1$) form a basis of $\underline{V}^{\otimes r}$ where $\eta_{n+1}:=\eta$. For any $I \subseteq \underline{r}$, if $i \in I$, we set $V_i:=V$, otherwise $V_i:=\mathbb{C}\eta$. Now we define a subspace $\underline{V}_I^{\otimes r}$ of $\underline{V}^{\otimes r}$ as $\underline{V}_I^{\otimes r}:=V_1\otimes \cdots \otimes V_r$. In particular, $\underline{V}_\phi^{\otimes r}=\mathbb{C}\eta\otimes \cdots \otimes \mathbb{C}\eta=\mathbb{C}\eta\otimes \cdots \otimes \eta$ with $r$ copies (Here $\phi$ is the empty set). Suppose $I=\{i_1,\cdots,i_l\} $ where $i_1<i_2<\cdots<i_l$, it is not hard to see that $\underline{V}_I^{\otimes r}$ has a basis as follows 
\begin{align}\label{space basis}
\{\eta_\textbf{j}:=\eta_{j_1} \otimes \cdots \otimes \eta_{j_r} \mid j_{i_k} \in \underline{n},k=1,\cdots,l;j_{d}=n+1 \text{ for } d \neq i_k\}.
\end{align}
Let $\underline{V}_l^{\otimes r}:=\bigoplus\limits_{I \subseteq \underline{r},~ \sharp I=l}\underline{V}_I^{\otimes r}$, thus $\underline{V}^{\otimes r}$ has a decomposition as a linear space $\underline{V}^{\otimes r}=\bigoplus_{l=0}^r \underline{V}_l^{\otimes r}$. 

Clearly, each $\underline{V}^{\otimes r}_l$ is stabilized under the action of $\frak S_r$. Thus this action gives rise to a representation of $\frak S_r$ on $\underline{V}^{\otimes r}_l$, denoted by $\Psi|_l$. Similarly, each $\underline{V}^{\otimes r}_l$ and $\underline{V}^{\otimes r}_I$ with $0 \leq l \leq r, I \subseteq \underline{r}$ are stabilized under the action of $\GL(V) \times \textbf{G}_m$. Hence they give rise to the representations of $\GL(V) \times \textbf{G}_m$ on $\underline{V}^{\otimes r}_l$ and $\underline{V}^{\otimes r}_I$, denoted by $\Phi|_l$ and $\Phi|_I$ respectively.

By classical Schur-Weyl duality, we have
\begin{equation}\label{SWD}
\text{End}_{\GL(V)}(V^{\otimes r})=\mathbb{C}\Psi|_r(\frak S_r), \text{    }
\text{End}_{\frak S_r}(V^{\otimes r})=\mathbb{C}\Phi|_r(\frak \GL(V)).
\end{equation}
Here $\mathbb{C}\Psi|_r(\frak S_r)$ and $\mathbb{C}\Phi|_r(\frak \GL(V))$ are the subalgebras of $\text{End}_{\mathbb{C}}(V^{\otimes r})$ that generated by $\Psi|_r(\frak S_r)$ and $\Phi|_r(\frak \GL(V))$ respectively.

\subsection{Tensor Invariants of $\text{O}(V)$ and $\text{Sp}(V)$.}\label{1.3} From now on until the end of this chapter, we always assume that $G=\text{O}(V)$ or $G=\text{Sp}(V)$ (In this case, $\dim V$ is even), i.e. $G$ is of type $B$, $C$ or $D$. The above discussion shows that ${V}^{\otimes r}$ becomes a $G$-module via the representation $\Phi|_r$. In general, ${V}^{\otimes k}$ is a $G$-module with $k \in \mathbb{Z}_{>0}$.

Suppose $G$ leaves invariant a nondegenerate bilinear form $\omega$ on $V$. Fix a basis $\{f_p\}_{1 \leq p \leq n}$ for $V$ and let $\{f^p\}_{1 \leq p \leq n}$ be the dual basis for $V$ relative to the form $\omega$. Thus $\omega(f_p,f^q)=\delta_{pq}$.

Let $\mathcal{B}_r(\varepsilon n):=\text{End}_G(V^{\otimes r})$ where
\begin{align*}
\varepsilon=
\begin{cases}
1& \omega \text{ is symmetric,}\\
-1& \omega \text{ is skew-symmetric.}
\end{cases}
\end{align*}
Clearly, $\mathcal{B}_r(\varepsilon n)$ is a subalgebra of $\text{End}_\mathbb{C}(V^{\otimes r})$.

For any pair $1 \leqslant i \text{ } \textless \text{ } j \leqslant r$ we define the $ij$-contraction operator $C_{ij}: V^{\otimes r} \rightarrow V^{\otimes (r-2)}$ by 
\begin{align}\label{C action}
C_{ij}(v_1 \otimes...\otimes v_r):=\omega(v_i,v_j)v_1\otimes...\otimes \hat{v_i} \otimes...\otimes \hat{v_j} \otimes...\otimes v_r
\end{align}
and the $ij$-expansion operator $D_{ij}:V^{\otimes (r-2)} \rightarrow V^{\otimes r}$ by 
\begin{align}\label{D action}
D_{ij}(v_1 \otimes...\otimes v_{r-2}):=\sum_{p=1}^{n}v_1\otimes...\otimes \underbrace{f_p}_\text{i th} \otimes...\otimes \underbrace{f^p}_\text{j th} \otimes...\otimes v_{r-2}.
\end{align}
Set $\tau_{ij}=D_{ij}C_{ij} \in \text{End}_G(V^{\otimes r})$. If $u=v_1 \otimes...\otimes v_r$ with $v_i \in V$, the definition implies that 
\begin{align}\label{tau action}
\tau_{ij}(u)=\omega(v_i,v_j)\sum_{p=1}^{n}v_1\otimes...\otimes \underbrace{f_p}_\text{i th} \otimes...\otimes \underbrace{f^p}_\text{j th} \otimes...\otimes v_r.
\end{align}

Set $z_k:=\tau_{k,k+1}$ and $s_i$ is identified with the transposition $(i,i+1)$ in $\frak S_r$ with $1 \leq i \leq {r-1}$, then $\mathcal{B}_r(\varepsilon n)$ is generated by $\Psi|_r(s_1),...,\Psi|_r(s_{r-1})$ and $z_1,...,z_{r-1}$. Define $P_{r-1}:=\frac{1}{n}z_{r-1}$. We have the following results.

\begin{lemma}(\cite[Lemma 10.1.5]{GW})\label{GW1}
	We have $\tau_{ij}^2=n\tau_{ij}$ with $1 \leq i \le j \leq r$, thus $P_{r-1}^2=P_{r-1}$.
\end{lemma}

\begin{remark}
	This lemma implies that $\text{ker}(C_{ij})=\text{ker}(\tau_{ij})$ with $1 \leq i \le j \leq r$.
\end{remark}

\begin{corollary}(\cite[Theorem 10.1.6]{GW})\label{cor 1.3}
	The algebra $\mathcal{B}_r(\varepsilon n)$ is generated by the operators $\Psi|_r(s)$ for $s \in \frak S_r$ and the projection $P_{r-1}$.
\end{corollary}

\begin{proposition}\label{B dual}
For $G=\text{O}(V)$ or $\text{Sp}(V)$, the following duality holds:
\begin{equation}\label{B dual 1}
\text{End}_G(V^{\otimes r})=\mathcal{B}_r(\varepsilon n),
\end{equation}
\begin{equation}\label{B dual 2}
\text{End}_{\mathcal{B}_r(\varepsilon n)}(V^{\otimes r})=\mathbb{C}\Phi|_r(G).
\end{equation}
Here $\mathbb{C}\Phi|_r(G)$ is the subalgebras of $\text{End}_{\mathbb{C}}(V^{\otimes r})$ that generated by $\Phi|_r(G)$.
\end{proposition}
\begin{proof}
	Obviously, (\ref{B dual 1}) is from the definition of $\mathcal{B}_r(\varepsilon n)$. Since $G$ is reductive, \cite[Definition 3.3.1]{GW} and \cite[Theorem 4.1.13]{GW} implies (\ref{B dual 2}).
\end{proof}

Consider the set $X_r$ of all graphs obtained from the two rows of dots by connecting each dot with exactly one other dot. We call an element of $X_r$ a Brauer diagram. Let $x \in X_r$ and let $t$ be the number of edges in the diagram of $x$ that connect a dot in the top row with another dot in the top row (call such an edge a top bar).  The bottom row of $x$ also has $t$ analogous such edges (call them bottom bars) and we call $x$ an $t$-bar diagram. A Brauer diagram $x$ is called a normalized diagrams, if all the edges of $x$ connecting the top and bottom rows are vertical. We can use top bars (or bottom bars) to express the normalized Brauer diagrams. Hence a normalized $t$-bar Brauer diagrams can be expressed as $\{i_1,j_1\},...,\{i_t,j_t\}$.

$$
\put(-15,10){\tiny$\bullet$}\put(-15,38){\tiny$\bullet$}
\put(-30,45){\small$1$}\put(-30,0){\small$2$}
\put(-15,45){\small$3$}\put(-15,0){\small$4$}
\put(-30,10){\tiny$\bullet$}\put(-30,38){\tiny$\bullet$}
\put(0,23){$\cdots\cdots$}
\put(40,45){\small$2r-1$}\put(40,0){\small$2r$}
\put(40,10){\tiny$\bullet$}\put(40,38){\tiny$\bullet$}
\put(-40,-25){\small Figure 1. A two-row array.}
$$
\\[-10pt]

$$
\put(-10,10){\tiny$\bullet$}\put(-10,38){\tiny$\bullet$}
\put(-30,10){\tiny$\bullet$}\put(-30,38){\tiny$\bullet$}
\put(10,10){\tiny$\bullet$}\put(10,38){\tiny$\bullet$}
\put(30,10){\tiny$\bullet$}\put(30,38){\tiny$\bullet$}
\put(50,10){\tiny$\bullet$}\put(50,38){\tiny$\bullet$}
\put(-30,45){\small$1$}\put(-30,0){\small$2$}
\put(-10,45){\small$3$}\put(-10,0){\small$4$}
\put(10,45){\small$5$}\put(10,0){\small$6$}
\put(30,45){\small$7$}\put(30,0){\small$8$}
\put(50,45){\small$9$}\put(50,0){\small$10$}
\put(-30,30){\Huge$\smile$}
\put(30,30){\Huge$\smile$}
\put(-10,10){\Huge$\frown$}
\put(30,10){\Huge$\frown$}
\put(-28.5,10.7){\line(4,3){40}}
\put(-80,-25){\small Figure 2. A Brauer diagram with $r=5$.}
$$\\[-10pt]

Let $Z_{r,t}$ be the set of normalized $t$-bar Brauer diagrams in $X_r$, and set $Z_r=\bigcup\limits_{t=0}^{[r/2]}Z_{r,t}$. Clearly, $Z_r$ is the set that consists of all the normalized Brauer diagrams with $2r$ dots, i.e. $Z_r$ contains all the normalized Brauer diagrams in $X_r$. Let $x_0 \in X_r$ be the graph with each dot in the top row connected with the dot below it, then $Z_{r,0}=\{x_0\}$.

$$
\put(-10,10){\tiny$\bullet$}\put(-10,38){\tiny$\bullet$}
\put(-30,10){\tiny$\bullet$}\put(-30,38){\tiny$\bullet$}
\put(10,10){\tiny$\bullet$}\put(10,38){\tiny$\bullet$}
\put(30,10){\tiny$\bullet$}\put(30,38){\tiny$\bullet$}
\put(50,10){\tiny$\bullet$}\put(50,38){\tiny$\bullet$}
\put(-30,45){\small$1$}\put(-30,0){\small$2$}
\put(-10,45){\small$3$}\put(-10,0){\small$4$}
\put(10,45){\small$5$}\put(10,0){\small$6$}
\put(30,45){\small$7$}\put(30,0){\small$8$}
\put(50,45){\small$9$}\put(50,0){\small$10$}
\put(-28.5,10.7){\line(0,1){30}}
\put(-8.5,10.7){\line(0,1){30}}
\put(12.5,10.7){\line(0,1){30}}
\put(32.5,10.7){\line(0,1){30}}
\put(52.5,10.7){\line(0,1){30}}
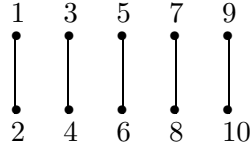
\put(-80,-25){\small Figure 3. Diagram for $x_0$ with $r=5$.}
$$\\[-10pt]

Now we list some useful results in \cite{GW} as follows.

\begin{lemma}(\cite[Lemma 10.1.2]{GW})\label{Brauer cal}
	Suppose that $z=\{i_1,j_1\},...,\{i_t,j_t\} \in X_r$ is a normalized $t$-bar Brauer diagram. Then
	\begin{align*}
	\tau_{i_pj_p}\tau_{i_qj_q}=\tau_{i_qj_q}\tau_{i_pj_p}, \text{ for } p \neq q.
	\end{align*}
	Thus the operator $\tau_z:=\prod_{p=1}^{t}\tau_{i_pj_p} \in \mathcal{B}_r(\varepsilon n)$ is defined independently of the order of the product.
\end{lemma}

Note that $\tau_{x_0}=1=\text{id}_{V^{\otimes r}}$. 

\begin{lemma}(\cite[Propsition 10.1.3]{GW})\label{Brauer span}
	Let $n={\rm{dim}}V$. The algebra $\mathcal{B}_r(\varepsilon n)$ is spanned by the set of operators $\Psi|_r(s)\tau_z$ with $s \in \frak S_r$ and $z \in Z_r$, i.e. 
	\begin{align*}
	\mathcal{B}_r(\varepsilon n)=\sum_{s \in \frak S_r}\sum_{z \in Z_r}\mathbb{C}\Psi|_r(s)\circ\tau_z.
	\end{align*}
\end{lemma}

\begin{lemma}(\cite[Corollary 10.1.4]{GW})\label{Brauer basis}
Suppose ${\rm{dim}}V=n \geq 2r$. Then the set $\{\Psi|_r(s)\tau_z \mid s \in \frak S_r, z \in Z_r \}$ is a basis for $\mathcal{B}_r(\varepsilon n)$, i.e.
\begin{align*}
\mathcal{B}_r(\varepsilon n)=\bigoplus_{s \in \frak S_r}\bigoplus_{z \in Z_r}\mathbb{C}\Psi|_r(s)\tau_z.
\end{align*}
\end{lemma}

\begin{remark}\label{Brauer dim}
	By \cite[Section 10.1.1]{GW} and \cite[Section 10.1.2]{GW}, there is a one-to-one correspondence between $\{\Psi|_r(s)\tau_z \mid s \in \frak S_r, z \in Z_r \}$ and $X_r$, namely $\Psi|_r(s)\tau_z$ ($s \in \frak S_r, z \in Z_r$) corresponds to a unique Brauer diagram with $2r$ dots. Thus $\sharp\{\Psi|_r(s)\tau_z \mid s \in \frak S_r, z \in Z_r \}=\sharp X_r=(2r-1)!!$, where $(2r-1)!!=1 \cdot 3 \cdots(2r-1)$. When $n \geq 2r$, we have ${\rm{dim}} (\mathcal{B}_r(\varepsilon n))=\sharp\{\Psi|_r(s)\tau_z \mid s \in \frak S_r, z \in Z_r \}=(2r-1)!!$.
\end{remark}

\subsection{Enhanced Brauer algebras.}\label{algebra B}
Fix $\sigma \in \text{End}_{\mathbb{C}}(V^{\otimes l})$ with $1 \leq l \leq r$, then $\sigma \otimes (\text{id}_{\mathbb{C}\eta})^{\otimes (r-l)}=\sigma \otimes \text{id}_{\mathbb{C}\eta^{\otimes (r-l)}} \in \text{End}_{\mathbb{C}}(\underline{V}_{\underline{l}}^{\otimes r})$. Next we extend $\sigma \otimes (\text{id}_{\mathbb{C}\eta})^{\otimes (r-l)}$ to an element $\sigma^{\underline{l}}$ of $\text{End}_{\mathbb{C}}(\underline{V}_l^{\otimes r})$ by annihilating any other summand $\underline{V}_I^{\otimes r}$ with $I \neq \underline{l}$.

In general, for $\eta_\textbf{j} \in \underline{V}_I^{\otimes r}$ with $\phi \neq I \subseteq \underline{r}$ and $\sharp I=l$, we can write
\begin{align*}
\eta_\textbf{j}=\Psi(\tau_I)\eta_{(k_1,...k_l,\underbrace{n+1,...,n+1}_\text{\textit{(r-l)} copies})} \text{ for some } (k_1,...,k_l) \in \underline{n}^l \text{ and }\tau_I \in \mathfrak{S}_r.
\end{align*}
  Then \begin{align}\label{def of action}
  \text{$\forall \sigma \in \text{End}_\mathbb{C}(V^{\otimes l})$ with $1 \leq l \leq r$,  define } \sigma^I:=\Psi(\tau_I) \circ \sigma^{\underline{l}} \circ \Psi(\tau_I^{-1}).
  \end{align}
  Clearly, $\sigma^I(\underline{V}_I^{\otimes r}) \subseteq \underline{V}_I^{\otimes r}$, hence $\sigma^I|_{\underline{V}_I^{\otimes r}} \in  \text{End}_\mathbb{C}{(\underline{V}_I^{\otimes r})}$. While for $J \subseteq \underline{r}$ with $J \neq I$, $\sigma^I(\underline{V}_J^{\otimes r})=0$. All ${\sigma}^I$ with $\sigma \in \mathcal{B}_l(\varepsilon n)$ and $1 \leq \sharp I=l \leq r$ constitute a subalgebra in $\text{End}_\mathbb{C}(\underline{V}^{\otimes r})$ which is denoted by $\mathcal{B}_l^{(I)}(\varepsilon n)$. 
Next we will define some subalgebras in $\text{End}_\mathbb{C}(\underline{V}^{\otimes r})$ that will be used in the following text.

\begin{defn}
	Let $\mathcal{B}(\varepsilon,n,r)$ be a subalgebra of $\text{End}_\mathbb{C}(\underline{V}^{\otimes r})$ generated by $\{\mathcal{B}_l^{(I)}(\varepsilon n) \mid  1 \leq l \leq r, I \subseteq \underline{r} \text{ satisfying } \sharp I=l\}$ and $\Psi(\frak S_r)$. Obviously, $\mathcal{B}(\varepsilon,n,r)$ is a finite dimensional associative algebra, which we call the Enhanced Brauer algebra over $\mathbb{C}$.
\end{defn}

 Similarly, we can get a subalgebra $\widetilde{\mathcal{B}}(\varepsilon,n,r)_l$ ($1 \leq l \leq r$) which is generated by $\{\sigma^I|_{\underline{V}_l^{\otimes r}} \mid \sigma \in \mathcal{B}_l(\varepsilon n),I \subseteq \underline{r} \text{ with } \sharp I=l\}$ and $\Psi|_l(\frak S_r)$. Let $\widetilde{\mathcal{B}}(\varepsilon,n,r)_0$ be an subalgebra spanned by (generated by) $\Psi|_0(\frak S_r)=\{\text{id}_{\mathbb{C}\eta^{\otimes r}}\}$, i.e. $\widetilde{\mathcal{B}}(\varepsilon,n,r)_0=\mathbb{C}\text{id}_{\mathbb{C}\eta^{\otimes r}}$.

Each element in $\widetilde{\mathcal{B}}(\varepsilon,n,r)_l$ ($0 \leq l \leq r$) can extend to an element in $\text{End}_\mathbb{C}(\underline{V}^{\otimes r})$ by annihilating any other summand $\underline{V}_k^{\otimes r}$ with $k \neq l$, thus we get a subalgebra of $\text{End}_\mathbb{C}(\underline{V}^{\otimes r})$ which is denoted by $\mathcal{B}(\varepsilon,n,r)_l$ ($0 \leq l \leq r$). Hence $\mathcal{B}(\varepsilon,n,r)_l$ ($1 \leq l \leq r$) is generated by $\{\sigma^I \mid  \sigma \in \mathcal{B}_l(\varepsilon n),  I \subseteq \underline{r} \text{ with } \sharp I=l\}$) and $\underline{\Psi}|_l(\frak S_r)$; $\mathcal{B}(\varepsilon,n,r)_0$ is spanned by (generated by) $\underline{\Psi}|_0(\frak S_r)$, i.e. $\mathcal{B}(\varepsilon,n,r)_0=\mathbb{C}\underline{\Psi}|_0(\frak S_r)$. Here $\underline{\Psi}|_l(\frak S_r) \subseteq \text{End}_\mathbb{C}(\underline{V}^{\otimes r})$ is the extension of $\Psi|_l(\frak S_r)$ by annihilating any other summand $\underline{V}_k^{\otimes r}$ with $k \neq l$. Let $\mathcal{B}_l^{(l)}(\varepsilon n)$ ($1 \leq l \leq r$) be a subalgebra of $\text{End}_\mathbb{C}(\underline{V}^{\otimes r})$ which is generated by $\mathcal{B}_l^{(I)}(\varepsilon n)$ with $I \subseteq \underline{r}$ and $\sharp I=l$.

For $1 \leq l \leq r$, let $\Psi_l: \frak S_l \rightarrow \GL(V^{\otimes l})$ with \begin{equation}\label{Sl}
\Psi_l(\sigma)(v_1 \otimes v_2 \otimes \cdots \otimes v_l)=v_{\sigma^{-1}(1)} \otimes v_{\sigma^{-1}(2)} \otimes \cdots \otimes v_{\sigma^{-1}(l)}.
\end{equation} Where $v_1, \cdots, v_l \in V$ and $\sigma \in  \frak S_l$. Then $\Psi_l$ is a representation of $\frak S_l$. Obviously, $\Psi_r=\Psi|_r$.

Denote by $D(n,r)$ the subalgebra of $\text{End}_\mathbb{C}(\underline{V}^{\otimes r})$ generated by $\Psi(\frak S_r)$ and $\{(\Psi_l(\sigma))^I \mid I \subseteq \underline{r}, \sigma \in \frak S_l \text{ with }l=\sharp I\}$. Let $D(n,r)_l$ ($1 \leq l \leq r$) be the subalgebra of $\text{End}_\mathbb{C}(\underline{V}^{\otimes r})$ generated by $\underline{\Psi}|_l(\frak S_r)$ and $\{(\Psi_l(\sigma))^I \mid I \subseteq \underline{r}\text{ with }\sharp I=l, \sigma \in \frak S_l\}$; while $D(n,r)_0$ is spanned by (generated by) $\underline{\Psi}|_0(\frak S_r)$, i.e. $D(n,r)_0=\mathbb{C}\underline{\Psi}|_0(\frak S_r)$. It is easy to see that $D(n,r)$ and $D(n,r)_l$ defined here are coincide with those defined in \cite[Definition 3.4]{BYY}. Hence \cite[Lemma 3.5(1)]{BYY} yields that $D(n,r)=\bigoplus_{l=0}^{r}D(n,r)_l$.

Keep in mind that $D(n,r)$ and $\mathcal{B}(\varepsilon,n,r)$ are subalgebras of $\text{End}_\mathbb{C}(\underline{V}^{\otimes r})$, hence $D(n,r)$ and $\mathcal{B}(\varepsilon,n,r)$ are finite dimensional associative algebras.

\begin{lemma}\label{B decom}
	$\mathcal{B}(\varepsilon,n,r)=\bigoplus\limits_{l=0}^r {\mathcal{B}}(\varepsilon,n,r)_l$, i.e. $\mathcal{B}(\varepsilon,n,r)=\bigoplus\limits_{l=0}^r \widetilde{\mathcal{B}}(\varepsilon,n,r)_l$.
\end{lemma}

\begin{proof}
	By \cite[Lemma 3.5(1)]{BYY}, we get $\underline{\Psi}|_l(\frak S_r) \subseteq {\mathcal{B}}(\varepsilon,n,r)$. Hence $\mathcal{B}(\varepsilon,n,r)_l \subseteq \mathcal{B}(\varepsilon,n,r)$ which yields that $\bigoplus\limits_{l=0}^r \mathcal{B}(\varepsilon,n,r)_l \subseteq \mathcal{B}(\varepsilon,n,r)$. On the other hand, it is clear that $\mathcal{B}(\varepsilon,n,r) \subseteq \bigoplus\limits_{l=0}^r \mathcal{B}(\varepsilon,n,r)_l$.
\end{proof}

From now on, we no longer need to distinguish $\widetilde{\mathcal{B}}(\varepsilon,n,r)_l$ from ${\mathcal{B}}(\varepsilon,n,r)_l$.

For the symmetric group $\frak S_r$, we denote by $(i,i+1)$ for
$i=1,\cdots, r-1$, the transposition just interchanging $i$ and $i+1$, and fixing the others.

Given $\sigma \in \text{End}_\mathbb{C}(\underline{V}^{\otimes r})$ and $I \subseteq \underline{r}$, define $\sigma^{[I]} \in \text{End}_\mathbb{C}(\underline{V}^{\otimes r})$ as below:
\begin{align}\label{def res}
\sigma^{[I]}|_{\underline{V}_J^{\otimes r}}:=\left\{
\begin{aligned}
\sigma|_{\underline{V}_I^{\otimes r}},  \text{ } I=J; \\
0, \text{ }, I \neq J.
\end{aligned}
\right.
\end{align}
By definition, $\sigma^{[I]}(\underline{V}_l^{\otimes r}) \subseteq \underline{V}_l^{\otimes r}$.
For $I=\{i_1, \cdots i_l\}$ and $J=\{j_1, \cdots j_l\}$ with $i_1 < \cdots <i_l$, $j_1 < \cdots < j_l$ and $0 \leq l \leq r$, there exists $\varepsilon_{J,I} \in \frak S_r$ satisfying $\varepsilon_{J,I}(i_t)=j_t$ for all $1 \leq t \leq l$. Now we define \begin{align}\label{E element}
E_{J,I}:=(\Psi(\varepsilon_{J,I}))^{[I]},
\end{align} and note that $E_{\phi,\phi}={({\rm{id}}_{\underline{V}^{\otimes r}})}^{[\phi]}$. Although the element $\varepsilon_{J,I}$ in $\frak S_r$ satisfying $\varepsilon_{J,I}(i_t)=j_t$ ($\forall 1 \leq t \leq l$) is not unique, $(\Psi(\varepsilon_{J,I}))^{[I]}$ does not depend on the choice of $\varepsilon_{J,I}$. It is clear that $E_{J,I}(\underline{V}_I^{\otimes r}) \subseteq \underline{V}_J^{\otimes r}$.
For each $\phi \neq I \subseteq \underline{r}$, we can see that $(\text{id}_{\underline{V}^{\otimes r}})^{[I]}=(\text{id}_{V^{\otimes l}})^I$ where $l=\sharp I>0$. Clearly, $\mathcal{B}(\varepsilon,n,r)_0=D(n,r)_0=\mathbb{C}\underline{\Psi}|_0(\frak S_r)=\mathbb{C}(\text{id}_{\underline{V}^{\otimes r}})^{[\phi]}=\mathbb{C}E_{\phi,\phi}$.

\begin{lemma}\label{calculation}\label{calculation'}
	For any $\sigma \in \text{End}_\mathbb{C}(V^{\otimes l})$, we have $E_{I,J} \circ \sigma^J=\sigma^I \circ E_{I,J}$ with $0 \leq \sharp I=\sharp J=l \leq r$.
\end{lemma}

\begin{corollary}\label{B span 1}
	The algebra $\mathcal{B}(\varepsilon,n,r)_l$ with $1 \leq l \leq r$ is spanned by the set $\{E_{IJ} \circ \sigma^J \mid  I \subseteq \underline{r},J \subseteq \underline{r},\sharp I=\sharp J=l, \sigma \in \mathcal{B}_l(\varepsilon n)\}$.
\end{corollary}

\begin{proof}
	Note that $\mathcal{B}(\varepsilon,n,r)_l$ is generated by $\sigma^J$ ($\forall \sigma \in \mathcal{B}_l(\varepsilon n), \forall J \subseteq \underline{r} \text{ with } \sharp J=l$) and $\underline{\Psi}|_l(\frak S_r)$. Since $\underline{\Psi}|_l(\frak S_r) \subseteq D(n,r)_l$, then \cite[Lemma 3.5(3)]{BYY} implies that $\underline{\Psi}|_l(s)$ ($s \in \frak S_r$) is the sum of $E_{I,J} \circ (\Psi_l(\tau))^J$ with $\sharp I=\sharp J=l$ and $\tau \in \frak S_l$ (Note that $(\Psi_l(\tau))^J$ is defined by (\ref{def of action})). Obviously, we have $\Psi_l(\tau)\sigma \in \mathcal{B}_l(\varepsilon n)$, which implies that $(\Psi_l(\tau))^J \sigma^J=(\Psi_l(\tau)\sigma)^J \in \mathcal{B}^{(J)}_l(\varepsilon n)$. Hence we just need to consider $E_{I,J} \circ \mathcal{B}^{(J)}_l(\varepsilon n)$ and $\mathcal{B}^{(J)}_l(\varepsilon n) \circ E_{I,J}$. Thanks to Lemma \ref{calculation}, we have $\mathcal{B}(\varepsilon,n,r)_l \subseteq \underset{\substack{I \subseteq \underline{r} \\ \lvert I \rvert=l}}\sum \text{ }\underset{\substack{J \subseteq \underline{r} \\ \lvert J \rvert=l}}\sum \underset{\substack{\sigma \in \mathcal{B}_l(\varepsilon n)}}\sum \mathbb{C} E_{IJ} \circ \sigma^J$.
	
	On the other hand, $E_{I,J}=(\Psi(\varepsilon_{I,J}))^{[J]} \in \mathcal{B}(\varepsilon,n,r)_l$ yields that $\underset{\substack{I \subseteq \underline{r} \\ \lvert I \rvert=l}}\sum \text{ }\underset{\substack{J \subseteq \underline{r} \\ \lvert J \rvert=l}}\sum \underset{\substack{\sigma \in \mathcal{B}_l(\varepsilon n)}}\sum \mathbb{C} E_{IJ} \circ \sigma^J \subseteq \mathcal{B}(\varepsilon,n,r)_l$, i.e. $\mathcal{B}(\varepsilon,n,r)_l= \underset{\substack{I \subseteq \underline{r} \\ \lvert I \rvert=l}}\sum \text{ }\underset{\substack{J \subseteq \underline{r} \\ \lvert J \rvert=l}}\sum \underset{\substack{\sigma \in \mathcal{B}_l(\varepsilon n)}}\sum \mathbb{C} E_{IJ} \circ \sigma^J$.
\end{proof}

Corollary \ref{B span 1} implies that for $1 \leq l \leq r$, we have \begin{align}\label{B_l 1}
\mathcal{B}(\varepsilon,n,r)_l=\underset{\substack{I \subseteq \underline{r} \\ \lvert I \rvert=l}}\bigoplus \text{ }\underset{\substack{J \subseteq \underline{r} \\ \lvert J \rvert=l}}\bigoplus E_{IJ} \circ \mathcal{B}^{(J)}_l(\varepsilon n).
\end{align} 

Keep in mind that $\forall z \in Z_l$ (i.e. $z$ is a normalized Brauer diagram with $2l$ dots), $\tau_z \in \text{End}_\mathbb{C}(V^{\otimes l})$ is defined by (\ref{tau action}) and Lemma \ref{Brauer cal}.

\begin{corollary}\label{B span 3}
	The algebra $\mathcal{B}(\varepsilon,n,r)_l$ with $1 \leq l \leq r$ is spanned by set $\{{E}_{IJ} (\Psi_l(s))^J (\tau_z)^J \mid  s \in \frak S_l,z \in Z_l,I \subseteq \underline{r},J \subseteq \underline{r},\sharp I=\sharp J=l\}$.
\end{corollary}

\begin{proof}
The result follows from Corollary \ref{B span 1} and Lemma \ref{Brauer span}.
\end{proof}

\begin{proposition}\label{B basis}
	Suppose $\frak B$ is a basis of $\mathcal{B}_l(\varepsilon n)$. For $l \in \underline{r}$,
	$\mathcal{B}(\varepsilon,n,r)_l$ has a basis $\{E_{IJ} \circ b^J \mid  b \in \frak B,I \subseteq \underline{r},J \subseteq \underline{r},\sharp I=\sharp J=l\}$.
\end{proposition}

\begin{proof}
	Since $\underline{V}^{\otimes r}=\bigoplus\limits_{I \subseteq \underline{r}}\underline{V}_I^{\otimes r}$, then for $I \subseteq \underline{r}$, we can define a linear transformation $\text{Pr}_I:\underline{V}^{\otimes r} \rightarrow \underline{V}_I^{\otimes r}$ by $\text{Pr}_I(\sum\limits_{J \subseteq \underline{r}}v_J)=v_I$ where $v_J \in \underline{V}_J^{\otimes r}$, i.e. $\text{Pr}_I$ is a projection operator.
	
	Suppose $\underset{\substack{I \subseteq \underline{r} \\ \lvert I \rvert=l}}\sum \text{ }\underset{\substack{J \subseteq \underline{r} \\ \lvert J \rvert=l}}\sum \sum\limits_{b \in \frak B} k_{(I,J,b)} E_{IJ} b^J=0$ where $k_{(I,J,b)} \in \mathbb{C}$. Since $(E_{IJ}  \circ b^J)(\underline{V}_J^{\otimes r}) \subseteq \underline{V}_I^{\otimes r}$, then for $S,T \subseteq \underline{r}$, it is clear that
	\begin{align*}
	\text{Pr}_S(E_{IJ}  \circ b^J)|_{\underline{V}_T^{\otimes r}}=\delta_{SI}\delta_{JT} (E_{IJ}  \circ b^J)|_{\underline{V}_J^{\otimes r}}=\begin{cases}
	{E}_{IJ} \circ (b^J)|_{\underline{V}_J^{\otimes r}}& , S=I \text{ and } J=T\\
	0& , S \neq I \text{ or } J \neq T.
	\end{cases}
	\end{align*}
	
	Thus we have
	\begin{align*}
	\text{Pr}_S(\underset{\substack{I \subseteq \underline{r} \\ \lvert I \rvert=l}}\sum \text{ }\underset{\substack{J \subseteq \underline{r} \\ \lvert J \rvert=l}}\sum \sum\limits_{b \in \frak B} k_{(I,J,b)} E_{IJ} b^J)|_{\underline{V}_T^{\otimes r}}=\underset{\substack{I \subseteq \underline{r} \\ \lvert I \rvert=l}}\sum \text{ }\underset{\substack{J \subseteq \underline{r} \\ \lvert J \rvert=l}}\sum \sum\limits_{b \in \frak B} k_{(I,J,b)} \text{Pr}_S(E_{IJ}  \circ b^J)|_{\underline{V}_T^{\otimes r}}
	\end{align*}
	\begin{align}\label{basis eq}
	=\sum\limits_{b \in \frak B} k_{(S,T,b)} E_{ST}  \circ (b^T)|_{\underline{V}_T^{\otimes r}}
	\end{align}
	
	Then 
	\begin{align*}
\underset{\substack{I \subseteq \underline{r} \\ \lvert I \rvert=l}}\sum \text{ }\underset{\substack{J \subseteq \underline{r} \\ \lvert J \rvert=l}}\sum \sum\limits_{b \in \frak B} k_{(I,J,b)} E_{IJ} b^J=0 \Longrightarrow \text{Pr}_S(\underset{\substack{I \subseteq \underline{r} \\ \lvert I \rvert=l}}\sum \text{ }\underset{\substack{J \subseteq \underline{r} \\ \lvert J \rvert=l}}\sum \sum\limits_{b \in \frak B} k_{(I,J,b)} E_{IJ} b^J)|_{\underline{V}_T^{\otimes r}}=0
	\end{align*}
	\begin{align*}
	\overset{\text{By }(\ref{basis eq})}\Longrightarrow \sum\limits_{b \in \frak B} k_{(S,T,b)} E_{ST}  \circ (b^T)|_{\underline{V}_T^{\otimes r}}=0 \Longrightarrow \sum\limits_{b \in \frak B} k_{(S,T,b)} b=0.
	\end{align*}
	
	This yields that $k_{(S,T,b)}=0$ for all $b \in \frak B$. By the arbitrariness of $S$ and $T$, we get $k_{(S,T,b)}=0$ for all $b \in \frak B$, $S,T \subseteq \underline{r}$ with $\lvert S \rvert=\lvert T \rvert=l$. This yields that all elements in $\{E_{IJ} \circ b^J \mid  b \in \frak B,z \in Z_l,I \subseteq \underline{r},J \subseteq \underline{r},\sharp I=\sharp J=l\}$ are linear independent.
\end{proof}

\begin{proposition}\label{B basis 3}
	Suppose $n \geq 2r$. Then $\{E_{IJ} (\Psi_l(s))^J (\tau_z)^J \mid  s \in \frak S_l,z \in Z_l,I \subseteq \underline{r},J \subseteq \underline{r},\sharp I=\sharp J=l\}$ is a basis of $\mathcal{B}(\varepsilon,n,r)_l$ with $l \in \underline{r}$,
	Thus $\text{dim}\left(\mathcal{B}(\varepsilon,n,r)_l\right)=\dbinom{r}{l}^2(2l-1)!!$ with $1 \leq l \leq r$, where $(2l-1)!!=(2l-1)\cdot (2l-3) \cdots 1$.
\end{proposition}

\begin{proof}
	Suppose $1 \leq l \leq r$. Note that $(\Psi_l(s)\tau_z)^J=(\Psi_l(s))^J (\tau_z)^J$. By Proposition \ref{B basis} and Lemma \ref{Brauer basis}, we see that $\{E_{IJ} (\Psi_l(s))^J (\tau_z)^J \mid  s \in \frak S_l,z \in Z_l,I \subseteq \underline{r},J \subseteq \underline{r},\sharp I=\sharp J=l\}$ is a basis of $\mathcal{B}(\varepsilon,n,r)_l$.
	
	Thanks to Remark \ref{Brauer dim}, we have $\sharp \{\Psi_l(s)\tau_z \mid s \in \frak S_l, z \in Z_l \}=(2l-1)!!$. Hence $\text{dim}\left(\mathcal{B}(\varepsilon,n,r)_l\right)=\sharp \{E_{IJ} (\Psi_l(s))^J (\tau_z)^J \mid  s \in \frak S_l,z \in Z_l,I \subseteq \underline{r},J \subseteq \underline{r},\sharp I=\sharp J=l\}=\sharp \{I \subseteq \underline{r} \mid \sharp I=l\} \cdot \sharp \{J \subseteq \underline{r} \mid \sharp J=l\}\cdot \sharp \{\Psi_l(s)\tau_z \mid s \in \frak S_l, z \in Z_l \}=\binom{r}{l}^2(2l-1)!!$.
\end{proof}

Keep in mind that $\mathcal{B}(\varepsilon,n,r)_0=\mathbb{C}E_{\phi,\phi}$, then $\text{dim}(\mathcal{B}(\varepsilon,n,r)_0)=1$.  

\begin{remark}
	Assume that $\sigma,\lambda \in \mathcal{B}_l(\varepsilon n)$ and $I,J,K,L \subseteq \underline{r}$ with $1 \leq l \leq r, \sharp J=\sharp K=\sharp L=\sharp I=l$. By Lemma \ref{calculation}, we have $(E_{IJ}\sigma^J)\cdot(E_{JL}\lambda^L)=E_{IJ}E_{JL}\sigma^L\lambda^L=E_{IL}(\sigma\lambda)^L$. If $J \neq K$, it is clear that $(E_{IJ}\sigma^J)\cdot(E_{KL}\lambda^L)=0$. Hence we have a multiplication formula as follows:
	\begin{align}\label{muli formula}
	(E_{IJ}\sigma^J)\cdot(E_{KL}\lambda^L)=\delta_{JK}E_{IL}(\sigma\lambda)^L.
	\end{align}
	
	 Let $\sigma=\Psi_l(s_1)\tau_{z_1}$ and $\lambda=\Psi_l(s_2)\tau_{z_2}$, where $s_1,s_2 \in \frak S_r$ and $z_1,z_2 \in Z_l$. Then $\sigma\lambda$ can be calculated by \cite[Lemma 10.1.5]{GW}.
\end{remark}

Keep in mind that $\{f_p\}_{1 \leq p \leq n}$ is a basis for $V$ and $\{f^p\}_{1 \leq p \leq n}$ is the dual basis for $V$ relative to the form $\omega$.
Given $1 \leq i < j \leq r$ and $I \subseteq \underline{r}$, if $\{i,j\} \subseteq I$, then $\tau_{ij} \in \mathcal{B}_r(\varepsilon n)$ acts on $\underline{V}_I^{\otimes r}$ like (\ref{tau action}):
\begin{align}\label{tau action 2}
\rho_I(\tau_{ij})(u):=\omega(v_i,v_j)\sum_{p=1}^{n}v_1\otimes...\otimes \underbrace{f_p}_\text{i th} \otimes...\otimes \underbrace{f^p}_\text{j th} \otimes...\otimes v_r,
\end{align}
where $u=v_1 \otimes...\otimes v_r \in \underline{V}_I^{\otimes r}$.

Suppose that $z=\{i_1,j_1\},...,\{i_t,j_t\} \in X_r$ is a normalized $t$-bar Brauer diagram. Due to Lemma \ref{Brauer cal}, $\tau_z=\prod_{p=1}^{t}\tau_{i_pj_p} \in \mathcal{B}_r(\varepsilon n)$ is defined independently of the order of the product. If $e(z):=\{i_1,j_1,\cdots,i_t,j_t\} \subseteq I \subseteq \underline{r}$, $\tau_z$ acts on $\underline{V}_I^{\otimes r}$ as follows: If $u=v_1 \otimes...\otimes v_r \in \underline{V}_I^{\otimes r}$,
\begin{align}\label{tau action 3}
\rho_I(\tau_z)(u):=(\rho_I(\tau_{i_1j_1}) \cdots \rho_I(\tau_{i_tj_t}))(u),
\end{align}
which is defined independently of the order of the product. While $e(z)=\{i_1,j_1,\cdots,i_t,j_t\} \nsubseteq I$, $\tau_z$ acts trivally on $\underline{V}_I^{\otimes r}$, i.e. $\rho_I(\tau_z)(u):=0$, where $u \in \underline{V}_I^{\otimes r}$.

For $u \in \underline{V}_J^{\otimes r}$ with $J \neq I$, $\rho_I(\tau_z)(u):=0$. Namely, $\tau_z$ acts trivally on $\underline{V}_J^{\otimes r}$. Since $\underline{V}^{\otimes r}=\bigoplus\limits_{I \subseteq \underline{r}}\underline{V}_I^{\otimes r}$, $\rho_I(\tau_z)$ is a linear transformation on $\underline{V}^{\otimes r}$ satisfying $\rho_I(\tau_z)(\underline{V}_I^{\otimes r}) \subseteq \underline{V}_I^{\otimes r}$. 

For $s \in \frak S_r$ and $z \in Z_r$, we can define a linear transformation $\rho_I(\Psi|_r(s)\tau_z)$ on $\underline{V}^{\otimes r}$:
\begin{align}\label{tau action 4}
\rho_I(\Psi|_r(s)\tau_z):=\Psi(s)\rho_I(\tau_z).
\end{align}

Suppose $n=\text{dim}V \geq 2r$, for $x \in \mathcal{B}_r(\varepsilon n)$, (\ref{tau action 4}) and Lemma \ref{Brauer basis} give rise to a linear operator $\rho_I(x)$. Clearly, $\rho_I: \mathcal{B}_r(\varepsilon n) \rightarrow \text{End}_\mathbb{C}(\underline{V}^{\otimes r})$ is a linear map. Hence we can define $\rho_I(\mathcal{B}_r(\varepsilon n)):=\{\rho_I(x) \mid x \in \mathcal{B}_r(\varepsilon n)\}$ which a (linear) subspace of $\text{End}_\mathbb{C}(\underline{V}^{\otimes r})$. It is clear that $\rho_{\underline{r}}(x)=x$ for all $x \in \mathcal{B}_r(\varepsilon n)$.

The special case of (\ref{tau action 4}) is just
\begin{align}\label{special}
\rho_I(\Psi|_r(s))=(\Psi(s))^{[I]}.
\end{align}

\begin{lemma}\label{pre lemma}
	If $n \geq 2r$, $\mathcal{B}(\varepsilon,n,r)_l=\bigoplus\limits_{I \subseteq \underline{r} \atop \sharp I=l}\rho_I(\mathcal{B}_r(\varepsilon n))$ with $0 \leq l \leq r$.
\end{lemma}

\begin{proof}
	The case $l=0$ is clear. Now we assume that $l \in \underline{r}$.
	
	Fix $I \subseteq \underline{r}$ with $\sharp I=l$. For $s \in \frak S_r$ and $z \in Z_r$, $\rho_I(\Psi|_r(s)\tau_z)=\Psi(s)\rho_I(\tau_z)$. If $e(z) \nsubseteq I$, $\rho_I(\tau_z)=0$. Hence $\rho_I(\Psi|_r(s)\tau_z)=0 \in \mathcal{B}(\varepsilon,n,r)_l$. 
	
	Now suppose $e(z) \subseteq I$, there exists $t \in Z_l$ such that $\rho_I(\tau_z)=(\tau_t)^I$. Suppose $I=\{i_1, \cdots, i_l\}$ with $i_1 < \cdots <i_l$ and $s(I)=\{j_1, \cdots, j_l\}$ with $j_1 < \cdots < j_l$ and $s(I)=\{j_1, \cdots, j_l\}$ with $j_1 < \cdots < j_l$. Moreover, assume that $m_1, \cdots, m_l$ are different numbers in $\underline{l}$ satisfying $s(i_t)=j_{m_t}$ ($\forall t \in \underline{l}$). There exist $\sigma' \in \frak S_r$ satisfying $\sigma'(i_t)=i_{m_t}$ ($\forall 1 \leq t \leq l$) and set $\sigma=\begin{pmatrix}
	1 & \cdots & l \\
	m_1 & \cdots & m_l
	\end{pmatrix} \in \frak S_l$. Recall that $\varepsilon_{s(I),I}(i_t)=j_t$ for all $1 \leq t \leq l$ (i.e. $\varepsilon_{s(I),I}(i_{m_t})=j_{m_t}$ for all $1 \leq t \leq l$).
	It is not hard to see that $(\Psi(s))^{[I]}=(\Psi(\varepsilon_{s(I),I} \cdot \sigma'))^{[I]}=(\Psi(\varepsilon_{s(I),I}))^{[I]} \cdot (\Psi(\sigma'))^{[I]}=E_{s(I),I} \cdot (\Psi_l(\sigma))^I$. It follows from Proposition \ref{B basis 3} that $\rho_I(\Psi|_r(s)\tau_z)=\Psi(s)\rho_I(\tau_z)=(\Psi(s))^{[I]}\rho_I(\tau_z)=E_{s(I),I} \cdot (\Psi_l(\sigma))^I \cdot (\tau_t)^I \in \mathcal{B}(\varepsilon,n,r)_l$.
	
	$\rho_I(\Psi|_r(s)\tau_z) \in \mathcal{B}(\varepsilon,n,r)_l$ implies that $\rho_I(\mathcal{B}_r(\varepsilon n)) \subseteq \mathcal{B}(\varepsilon,n,r)_l$, i.e.  $\bigoplus\limits_{I \subseteq \underline{r} \atop \sharp I=l}\rho_I(\mathcal{B}_r(\varepsilon n)) \subseteq \mathcal{B}(\varepsilon,n,r)_l$.
	
	Conversely, fix $s \in \frak S_l,t \in Z_l$ and $I,J \subseteq \underline{r}$ with $\sharp I=\sharp J=l$, we consider $E_{I,J}(\Psi_l(s))^J(\tau_t)^J$. There exists $\sigma \in \frak S_r$ and $z' \in Z_r$ such that $(\Psi(\sigma))^{[J]}=(\Psi_l(s))^J$ and $\rho_J(\tau_{z'})=(\tau_t)^J$. Thus $E_{I,J}(\Psi_l(s))^J(\tau_t)^J=(\Psi(\varepsilon_{I,J}))^{[J]} \cdot (\Psi(\sigma))^{[J]}(\tau_t)^J=(\Psi(\varepsilon_{I,J} \cdot \sigma))^{[J]}(\tau_t)^J=\Psi(\varepsilon_{I,J} \cdot \sigma) \circ (\tau_t)^J=\Psi(\varepsilon_{I,J} \cdot \sigma) \circ \rho_J(\tau_{z'}) \in \rho_J(\mathcal{B}_r(\varepsilon n))$. Hence $E_{I,J}(\Psi_l(s))^J(\tau_t)^J \in \bigoplus\limits_{I \subseteq \underline{r} \atop \sharp I=l}\rho_I(\mathcal{B}_r(\varepsilon n))$. It follows from Proposition \ref{B basis 3} that $\mathcal{B}(\varepsilon,n,r)_l \subseteq \bigoplus\limits_{I \subseteq \underline{r} \atop \sharp I=l}\rho_I(\mathcal{B}_r(\varepsilon n))$.
\end{proof}

Lemma \ref{pre lemma} and Lemma \ref{B decom} implies that \begin{align}\label{for final use 2}
\mathcal{B}(\varepsilon,n,r)=\bigoplus\limits_{I \subseteq \underline{r}}\rho_I(\mathcal{B}_r(\varepsilon n)).
\end{align}

\section{Invariants and dualities of Levi and parabolic groups.}\label{2}

\subsection{Restricted dualities}\label{RTIs}
Keep the notations as before. 
Set $W_i:=\bigoplus_{t=i}^r\underline{V}_t^{\otimes r}$ with $0 \leq i \leq r$ and $W_{r+1}=0$. It is clear that $W_0=\underline{V}^{\otimes r}$ and $W_r=V^{\otimes r}$. Recall that $\{\eta_{1},\cdots,\eta_n\}$ and $\{\eta_{1},\cdots,\eta_n,\eta\}$ are basis of $V$ and $\underline{V}$ repectively. Namely, $V=\bigoplus_{i=1}^n\mathbb{C}\eta_{i}$ and $\underline{V}=V \oplus \eta$. For $\textbf{j}=\{j_1,\cdots,j_r\} \in (\underline{n+1})^r$ (i.e. $1 \leq j_1,\cdots,j_r \leq n+1$), we have $\eta_{\textbf{j}}=\eta_{j_1} \otimes \cdots \otimes \eta_{j_r}$.

\begin{proposition}\label{key decom} Let $G$ be a closed subgroup of $\GL(V)$. 
	For any $f \in \text{End}_{\underline{G}}(\underline{V}^{\otimes r})$ and $l \in \{0,1,\cdots, r\}$, we have $f(\underline{V}_l^{\otimes r}) \subseteq W_l$. Hence $f(W_l) \subseteq W_l$.
\end{proposition}

\begin{proof}
	Since Proposition \ref{key decom}(1) is valid for $l=r+1$ (note that $\underline{V}_{r+1}^{\otimes r}=0$). Now we suppose that Proposition \ref{key decom}(1) is valid for $l \geq k+1$ (with $0 \leq k \leq r$), we just need to show that Proposition \ref{key decom}(1) is valid when $l=k$ (by induction on $l$). Supoose $f \in \text{End}_{\underline{G}}(\underline{V}^{\otimes r})$. If $k+1 \leq l \leq r$, then $f(\underline{V}_l^{\otimes r}) \subseteq W_l$ by inductive hypotheses (note that $\underline{V}_{r+1}^{\otimes r}=0$). Now we just need to prove that $f(\underline{V}_k^{\otimes r}) \subseteq W_k$.
	 
	 For any nonzero $w \in \underline{V}_k^{\otimes r}$ and suppose $f(w) \neq 0$ as well. Keep in mind that all $\eta_\textbf{i}=\eta_{i_1} \otimes \cdots \otimes \eta_{i_r}$ for $\textbf{i}=\{i_1,\cdots,i_r\} \in (\underline{n+1})^r$ form a basis of $\underline{V}^{\otimes r}$ (see Section \ref{1.2}). Then we write
	 \begin{align}\label{2.1}
	 f(w)=\sum_{i=1}^t a_i \eta_{\textbf{j}_i} \in \underline{V}^{\otimes r}
	 \end{align}
	 with $\textbf{j}_i \in (\underline{n+1})^r$ and $a_{i} \in \mathbb{C}^{\times}$ ($1 \leq i \leq t$).
	 
	 By the assumption, $(f\circ \Phi(\underline{g}))(w)=(\Phi(\underline{g}) \circ f)(w)$ for any $\underline{g}:=(g,v) \in \underline{G}$ with $g \in G$ and $v \in V$. In particular, we take some special element $(e,v) \in \underline{G}$, denoted by $e^v$. Then we have an equation \begin{align}\label{commute}
	 	(f\circ \Phi(e^v))(w)=(\Phi(e^v) \circ f)(w).
	 \end{align}
	 
	 Suppose $w_1 \otimes w_2 \otimes \cdots \otimes w_k \otimes \eta^{r-k} \in \underline{V}_k^{\otimes r}$ with $w_q \in V$ ($\forall 1 \leq q \leq k$). Since $e^v(\eta)=v+\eta$ and $e^v.x=x$ for all $x \in V$ (by (\ref{def of module})), then 
	 \begin{align}\label{jisuan1}
	 \Phi(e^v)(w_1 \otimes w_2 \otimes \cdots \otimes w_k \otimes \eta^{r-k})=w_1 \otimes w_2 \otimes \cdots \otimes w_k \otimes (v+\eta)^{r-k}.
	 \end{align}
	Thanks to $f(\underline{V}_l^{\otimes r}) \subseteq W_l$ with $k+1 \leq l \leq r$, we conclude that 
	\begin{align}\label{jisuan2}
	f(w_1 \otimes w_2 \otimes \cdots \otimes w_k \otimes (v+\eta)^{r-k}) \in f(w_1 \otimes w_2 \otimes \cdots \otimes w_k \otimes \eta^{r-k})+W_{k+1}.
	\end{align}
	Note that $\underline{V}_l^{\otimes r}=0$ with $l>r$ and $W_{k+1}=\bigoplus_{l=k+1}^{r}\underline{V}_l^{\otimes r}=0$ with $k=r$.
	Hence (\ref{jisuan1}) implies that 
	\begin{align}\label{jisuan3}
	(f \circ \Phi(e^v))(w_1 \otimes w_2 \otimes \cdots \otimes w_k \otimes \eta^{r-k}) \in f(w_1 \otimes w_2 \otimes \cdots \otimes w_k \otimes \eta^{r-k})+W_{k+1}.
	\end{align}
	So (\ref{jisuan3}) yields that  
	 	\begin{align}\label{jisuan4}
	 (f \circ \Phi(e^v))(w) \in f(w)+W_{k+1}=\sum_{i=1}^t a_i \eta_{\textbf{j}_i}+\bigoplus_{l=k+1}^{r}\underline{V}_l^{\otimes r}.
	 \end{align}
	 
	 Obviously, (\ref{2.1}) implies that 
	 \begin{align}\label{jisuan5}
	 (\Phi(e^v)\circ f)(w))=\Phi(e^v)(f(w))=\sum_{i=1}^t a_i \Phi(e^v)(\eta_{\textbf{j}_i}).
	 \end{align}
	 By comparing (\ref{commute}), (\ref{jisuan4}) and (\ref{jisuan5}), we have 
	 \begin{align}\label{jisuan6}
	 \sum_{i=1}^t a_i \Phi(e^v)(\eta_{\textbf{j}_i}) \in \sum_{i=1}^t a_i \eta_{\textbf{j}_i}+\bigoplus_{l=k+1}^{r}\underline{V}_l^{\otimes r}.
	 \end{align}
	 Suppose $\textbf{j}_i=\{j_i^1,\cdots,j_i^r\} \in (\underline{n+1})^r$ with $1 \leq i \leq t$. For $1 \leq l \leq r$, let \begin{equation*} \widetilde{\eta}_{j_i^l}:=\begin{cases}{\eta}_{{j_i^l}}&\mbox{,if ${\eta}_{{j_i^l}} \neq \eta$;}\\
	 v+\eta&\mbox{,if ${\eta}_{{j_i^l}}=\eta$.} 
	 \end{cases} 
	 \end{equation*}
	Set $\widetilde{\eta}_{\textbf{j}_i}:=\widetilde{\eta}_{{j_i^1}} \otimes \cdots \otimes \widetilde{\eta}_{j_i^r}$ with $1 \leq i \leq t$. Clearly, $e^v.{\eta}_{j_i^r}=\widetilde{\eta}_{j_i^r}$ (By (\ref{def of module})). Thus $\Phi(e^v)(\eta_{\textbf{j}_i})=\widetilde{\eta}_{\textbf{j}_i}$, so we have 
	\begin{align}\label{jisuan7}
	\sum_{i=1}^t a_i (\widetilde{\eta}_{\textbf{j}_i}-\eta_{\textbf{j}_i}) \in \bigoplus_{l=k+1}^{r}\underline{V}_l^{\otimes r}.
	\end{align}
	Without loss of generality, we can assume that $\eta_{\textbf{j}_1} \in \underline{V}_h^{\otimes r}$ with $0 \leq h <k$ (Here we assume that $k>0$; if $k=0$, $f(\underline{V}_k^{\otimes r}) \subseteq W_k$ is clear), i.e. $h=\sharp\{l \mid j_i^l \neq \eta \text{ with } 1 \leq l \leq r\}$. Set \begin{equation*} a_i^l:=\begin{cases}a_i&\mbox{,if ${\eta}_{{j_i^l}} \neq \eta$;}\\
	0&\mbox{,if ${\eta}_{{j_i^l}}=\eta$.} 
	\end{cases} 
	\end{equation*}
	with $1 \leq l \leq r$. Now assume that $I_l=\{-\sum_{i \in J}\frac{a_i^l}{a_1} \cdot {\eta}_{{j_i^l}} \mid J \subseteq \{2,3,\cdots,t\}\}$ with $1 \leq l \leq r$, $I_l$ is finite. Since $V$ is an infinite set and $\bigcup_{l=1}^rI_l$ is finite, then $V-\bigcup_{l=1}^rI_l$ is nonempty, i.e. $\exists v^* \in V-\bigcup_{l=1}^rI_l$. In the case $v=v^*$, we see that $\sum_{i=1}^t a_i (\widetilde{\eta}_{\textbf{j}_i}-\eta_{\textbf{j}_i}) \notin \bigoplus_{l=k+1}^{r}\underline{V}_l^{\otimes r}$, which contradicts (\ref{jisuan7}). So $\forall 1 \leq i \leq t$, we have $\eta_{\textbf{j}_i} \in \bigoplus_{l=k}^{r}\underline{V}_l^{\otimes r}=W_k$, hence $f(w)=\sum_{i=1}^t a_i \eta_{\textbf{j}_i} \in W_k$, i.e. $f(\underline{V}_k^{\otimes r}) \subseteq W_k$. By induction on $l$, we have $f(\underline{V}_l^{\otimes r}) \subseteq W_l$ ($\forall 0 \leq l \leq r+1$). Thus $f(\underline{V}_l^{\otimes r}) \subseteq W_l$ ($\forall 0 \leq l \leq r$).
	
	Since $f(\underline{V}_l^{\otimes r}) \subseteq W_l$ ($\forall 0 \leq l \leq r$), then $f(W_l) \subseteq W_l$ ($\forall 0 \leq l \leq r$).  
\end{proof}

Assume that $G$ is a subgroup of $\GL(V)$. Let $I,J \subseteq \underline{r}$ with $\sharp I=s$ and $\sharp J=t$ ($0 \leq s,t \leq r$). Obviously, we have $\underline{V}_I^{\otimes r} \cong V^{\otimes s}$ and $\underline{V}_J^{\otimes r} \cong V^{\otimes t}$ as isomorphisms of linear spaces. Hence, for $f \in \text{Hom}_\mathbb{C}(V^{\otimes s},V^{\otimes t})$, we can get $f^* \in \text{Hom}_\mathbb{C}(\underline{V}_I^{\otimes r},\underline{V}_J^{\otimes r})$. As a result, there is an isomorphism of vector spaces
\begin{align}\label{to Brauer case}
\text{Hom}_\mathbb{C}(V^{\otimes s},V^{\otimes t}) \xrightarrow{\cong}  \text{Hom}_\mathbb{C}(\underline{V}_I^{\otimes r},\underline{V}_J^{\otimes r}), \text{ with }f \mapsto f^*.
\end{align}

In fact, $\underline{V}_I^{\otimes r} \cong V^{\otimes s}$ and $\underline{V}_J^{\otimes r} \cong V^{\otimes t}$ are $G$-module isomorphisms. Hence for $f \in \text{Hom}_G(V^{\otimes s},V^{\otimes t})$, we have $f^* \in \text{Hom}_G(\underline{V}_I^{\otimes r},\underline{V}_J^{\otimes r})$. So there is an isomorphism of linear spaces: 
\begin{align}\label{to Brauer case 2}
T_{IJ}: \text{Hom}_G(V^{\otimes s},V^{\otimes t}) \xrightarrow{\cong} \text{Hom}_G(\underline{V}_I^{\otimes r},\underline{V}_J^{\otimes r}), \text{ with } T_{IJ}(f)=f^*.
\end{align}

By \cite[(4.1)]{GW}, we have the following isomorphisms of vector spaces
\begin{align}\label{G inv cong}
\text{Hom}_G(V^{\otimes s},V^{\otimes t}) \cong (V^{\otimes t} \otimes (V^{\otimes s})^*)^G \cong (V^{\otimes t} \otimes (V^*)^{\otimes s})^G.
\end{align}

From now on until the end of this section, we assume that $G=\text{O}(V)$ or $G=\text{Sp}(V)$ (In this case, $n=\dim V$ is even), and suppose $G$ leaves invariant a nondegenerate bilinear form $\omega$. Fix a basis $\{f_p:1 \leq p \leq n\}$ for $V$ and let $\{f^p:1 \leq p \leq n\}$ be the dual basis for $V$ relative to the form $\omega$. It is not hard to see that $\omega$ induces a $G$-module isomorphism $V \cong V^*$. Hence 
\begin{align}\label{G inv cong 2}
\text{Hom}_G(V^{\otimes s},V^{\otimes t}) \cong (V^{\otimes t} \otimes (V^*)^{\otimes s})^G \cong (V^{\otimes (s+t)})^G.
\end{align}

Combine (\ref{G inv cong 2}) with (\ref{to Brauer case 2}), the following isomorphism of linear spaces holds.
\begin{align}\label{yield 0}
\text{Hom}_G(\underline{V}_I^{\otimes r},\underline{V}_J^{\otimes r}) \cong (V^{\otimes (s+t)})^G.
\end{align}

\begin{lemma}(\cite[Theorem 5.3.3]{GW})\label{BCD inv}
	If $m$ is odd, them $(V^{\otimes m})^G=0$.
\end{lemma}

Given $I,J \subseteq \underline{r}$ satisfying $\sharp I=\sharp J$. Keep in mind that $E_{IJ}$ is defined in (\ref{E element}).

Convention: $E_{IJ}$ can be regarded as an element in $\text{Hom}_\mathbb{C}(\underline{V}_J^{\otimes r},\underline{V}_I^{\otimes r})$ with $\sharp I=\sharp J$, that is $E_{IJ}|_{\underline{V}_J^{\otimes r}}$. For convenience, we still use $E_{IJ}$ instead of $E_{IJ}|_{\underline{V}_J^{\otimes r}}$ if the context is clear. It is easy to check that $E_{IJ}$ is a $G$-module isomorphism. Moreover, we have $T^{-1}_{JI}(E_{IJ})=\text{id}_{V^{\otimes l}}$ (see \eqref{to Brauer case 2}). 

Given $I \subsetneq J \subseteq \underline{r}$, and assume that $\{i,j\}=J-I$ with $i <j$. Then $\sharp I+2=s+2=t=\sharp J$. We can define two maps as follows: 

$\forall c \in \mathbb{C}$. For $i \in I$, $v_i \in V$; for $i \notin I$, $v_i=\eta$. Then $v_1 \otimes \cdots \otimes v_r \in \underline{V}_I^{\otimes r}$. Now we define 
\begin{align*}
\mathcal{D}_{ij}^{IJ}:\underline{V}_I^{\otimes r} \rightarrow \underline{V}_J^{\otimes r}
\end{align*}
\begin{align}
cv_1 \otimes \cdots \otimes v_r \mapsto c\sum_{p=1}^n v_1 \otimes \cdots \otimes \underbrace{f_p}_{i\text{ th}} \otimes \cdots \otimes \underbrace{f^p}_{j\text{ th}} \otimes \cdots \otimes v_r. 
\end{align}
$\forall c \in \mathbb{C}$. For $i \in J$, $v_i \in V$; for $i \notin J$, $v_i=\eta$. Then $v_1 \otimes \cdots \otimes v_r \in \underline{V}_J^{\otimes r}$. Now we define
\begin{align*}
\mathcal{C}_{ij}^{IJ}:\underline{V}_J^{\otimes r} \rightarrow \underline{V}_I^{\otimes r}
\end{align*}
\begin{align}
cv_1 \otimes \cdots \otimes v_r \mapsto c\omega(v_i,v_j) v_1 \otimes \cdots \otimes \underbrace{\eta}_{i\text{ th}} \otimes \cdots \otimes \underbrace{\eta}_{j\text{ th}} \otimes \cdots \otimes v_r. 
\end{align}

By (\ref{to Brauer case}), $\mathcal{D}_{ij}^{IJ}$ can be viewed as a map from $V^{\otimes s}$ to $V^{\otimes t}$ which is denoted by $\mathcal{D}$. Using (\ref{D action}), there exists $1 \leq p<q \leq s$, such that $\mathcal{D}=D_{pq}$ (where $D_{pq}$ is defined by (\ref{D action})). Since $D_{pq}$ is a linear map, $\text{Hom}_\mathbb{C}(V^{\otimes s},V^{\otimes t}) \cong \text{Hom}_\mathbb{C}(\underline{V}_I^{\otimes r},\underline{V}_J^{\otimes r})$ yields that $\mathcal{D}_{ij}^{IJ}$ is also a linear map. Similarly, we can also conclude that $\mathcal{C}_{ij}^{IJ} \in \text{Hom}_\mathbb{C}(\underline{V}_J^{\otimes r},\underline{V}_I^{\otimes r})$. Moreover, it is easy to check that both of $\mathcal{C}_{ij}^{IJ}$ and $\mathcal{D}_{ij}^{IJ}$ are $G$-module homomorphisms. 

The following lemma is clear.
\begin{lemma}\label{def of CD 1}
	Assume that $I \subseteq J \subseteq \underline{r}$ satisfying $J-I= \{i_1,j_1,i_2,j_2\}$ with $i_1 < j_1$ and $i_2 < j_2$. Set $J_t=I \cup \{i_t,j_t\}$ with $t=1,2$. Then 
	\begin{align*}
	\mathcal{D}_{i_2j_2}^{J_1J}\mathcal{D}_{i_1j_1}^{IJ_1}=\mathcal{D}_{i_1j_1}^{J_2J}\mathcal{D}_{i_2j_2}^{IJ_2}, \text{    } \mathcal{C}_{i_2j_2}^{IJ_2}\mathcal{C}_{i_1j_1}^{J_2J}=\mathcal{C}_{i_1j_1}^{IJ_1}\mathcal{C}_{i_2j_2}^{J_1J}.
	\end{align*}
\end{lemma}

	Assume that $I \subseteq J \subseteq \underline{r}$ satisfying $J-I= \{i_1,j_1,\cdots,i_t,j_t\}$ with $i_p<j_p$ ($p=1,\cdots,t$). Let $\Pi:=\{\{i_1,j_1\},\cdots,\{i_t,j_t\}\}$. Set $J_0:=I$ and $J_s:=J_{s-1} \cup \{i_s,j_s\}$ with $s=1,\cdots,t$, then $J=J_t$. Lemma \ref{def of CD 1} impies that $\mathcal{D}_\Pi^{IJ}:=\mathcal{D}_{i_tj_t}^{J_{t-1}J_t} \cdots \mathcal{D}_{i_1j_1}^{J_0J_1}$ and $\mathcal{C}_\Pi^{IJ}:=\mathcal{C}_{i_1j_1}^{J_0J_1} \cdots \mathcal{C}_{i_tj_t}^{J_{t-1}J_t}$ can be defined independently of the order of
	the product. Since $\mathcal{C}_{i_sj_s}^{J_{s-1}J_s}$ and $\mathcal{D}_{i_sj_s}^{J_{s-1}J_s}$ are $G$-modules, then both of $\mathcal{C}_\Pi^{IJ}$ and $\mathcal{D}_\Pi^{IJ}$ are $G$-modules. In the case $\Pi=\phi$ (i.e. $I=J$), then $\mathcal{D}_\Pi^{IJ}=\mathcal{C}_\Pi^{IJ}=\text{id}_{\underline{V}_I^{\otimes r}}$.
	
	Now it is a position to investigate the structure of $\text{End}_G(\underline{V}^{\otimes r})$.
	
	 For $0 \leq s,t \leq r$, it is clear that $\text{Hom}_G(\underline{V}_s^{\otimes r},\underline{V}_t^{\otimes r})$ is a linear subspace of $\text{End}_G(\underline{V}^{\otimes r})$. In particular, $\text{End}_G(\underline{V}_s^{\otimes r})$ is a subalgebra of $\text{End}_G(\underline{V}^{\otimes r})$. 
	 
	 Keep in mind that $\underline{V}^{\otimes r}=\bigoplus_{l=0}^r \underline{V}_l^{\otimes r}$ and $\underline{V}_l^{\otimes r}=\bigoplus\limits_{I \subseteq \underline{r},~ \sharp I=l}\underline{V}_I^{\otimes r}$, then we decompose the algebra $\text{End}_G(\underline{V}^{\otimes r})$ as followings: 
	As vector spaces, 
\begin{align}\label{End decom 1}
\text{End}_G(\underline{V}^{\otimes r}) \cong \bigoplus_{s=0}^r\bigoplus_{t=0}^r\text{Hom}_G(\underline{V}_s^{\otimes r},\underline{V}_t^{\otimes r}),
\end{align}
\begin{align}\label{Hom decom 1}
\text{Hom}_G(\underline{V}_s^{\otimes r},\underline{V}_t^{\otimes r}) \cong \underset{I \subseteq \underline{r} \atop \sharp I=s}{\bigoplus}\underset{J \subseteq \underline{r} \atop \sharp J=t}{\bigoplus}\text{Hom}_G(\underline{V}_I^{\otimes r},\underline{V}_J^{\otimes r}).
\end{align}

Similarly, we have the following statements
\begin{align}\label{End decom 2}
\text{End}_{G \times \textbf{G}_m}(\underline{V}^{\otimes r}) \cong \bigoplus_{s=0}^r\bigoplus_{t=0}^r\text{Hom}_{G \times \textbf{G}_m}(\underline{V}_s^{\otimes r},\underline{V}_t^{\otimes r}),
\end{align}
\begin{align}\label{Hom decom 2}
\text{Hom}_{G \times \textbf{G}_m}(\underline{V}_s^{\otimes r},\underline{V}_t^{\otimes r}) \cong \underset{I \subseteq \underline{r} \atop \sharp I=s}{\bigoplus}\underset{J \subseteq \underline{r} \atop \sharp J=t}{\bigoplus}\text{Hom}_{G \times \textbf{G}_m}(\underline{V}_I^{\otimes r},\underline{V}_J^{\otimes r}).
\end{align}

Now fix $I,J \subseteq \underline{r}$, set $s=\sharp I$ and $t=\sharp J$. Due to (\ref{End decom 1}) and (\ref{Hom decom 1}), we just need to investigate $\text{Hom}_G(\underline{V}_I^{\otimes r},\underline{V}_J^{\otimes r})$.

\textbf{Case 1. $s+t$ is odd.} 

Note that $s+t$ is odd implies that $s \neq t$.

Thanks to Lemma \ref{BCD inv} and (\ref{yield 0}), we have
\begin{align}\label{B_{IJ} odd}
	\text{Hom}_G(\underline{V}_I^{\otimes r},\underline{V}_J^{\otimes r})=0.
\end{align}

Futhermore, (\ref{Hom decom 1}) yields that  \begin{align}\label{Bst odd}
\text{Hom}_G(\underline{V}_s^{\otimes r},\underline{V}_t^{\otimes r})=\underset{I \subseteq \underline{r} \atop \sharp I=s}{\bigoplus}\underset{J \subseteq \underline{r} \atop \sharp J=t}{\bigoplus}\text{Hom}_G(\underline{V}_I^{\otimes r},\underline{V}_J^{\otimes r})=0. 
\end{align}

Hence $\text{Hom}_{G \times \textbf{G}_m}(\underline{V}_s^{\otimes r},\underline{V}_t^{\otimes r})=0$.

\textbf{Case 2. $s=t$.}

Set $\mathcal{B}_{IJ}:=E_{JI}\cdot\mathcal{B}_s^{(I)}(\varepsilon n)=\{E_{JI}\circ\lambda \mid \lambda \in \mathcal{B}_s^{(I)}(\varepsilon n)\}$.
 In particular, $B_{II}=\mathcal{B}_s^{(I)}(\varepsilon n)$.

For any $f \in \text{End}_\mathbb{C}(V^{\otimes s})$, we can also use (\ref{def of action}) to define $f^I$. Obviously, for $K \subseteq \underline{r}$ with $K \neq I$, $f^I(\underline{V}_K^{\otimes r})=0$; while $f^I(\underline{V}_I^{\otimes r}) \subseteq \underline{V}_I^{\otimes r}$, thus $f^I|_{\underline{V}_I^{\otimes r}} \in \text{End}_\mathbb{C}(\underline{V}_I^{\otimes r})$.For convenience, we always use $f^I$ instead of $f^I|_{\underline{V}_I^{\otimes r}}$, if the context is clear.

By (\ref{to Brauer case 2}), $T_{II}$ gives an isomorphism of vector spaces between $\text{End}_G(V^{\otimes s})$ and $\text{End}_G(\underline{V}_I^{\otimes r})$. It is not hard to see that $\forall f \in \text{End}_G(V^{\otimes s})$, we have $T_{II}(f)=f^I$, thus $\text{End}_G(\underline{V}_I^{\otimes r})=\{f^I \mid f \in \text{End}_G(V^{\otimes s})\}$. 

Note that $\text{End}_G(V^{\otimes s})=\mathcal{B}_s(\varepsilon n)$, $\text{End}_G(\underline{V}_I^{\otimes r})=\{f^I \mid f \in \text{End}_G(V^{\otimes s})\}$ implies that
\begin{align}\label{G inv II}
\text{End}_G(\underline{V}_I^{\otimes r})=\mathcal{B}_s^{(I)}(\varepsilon n).
\end{align}

\begin{lemma}\label{B_II 2}
	Assume $s=t$, then $\text{Hom}_G(\underline{V}_I^{\otimes r},\underline{V}_J^{\otimes r})=\mathcal{B}_{IJ}.$
\end{lemma}
\begin{proof}
	For $f \in \text{Hom}_G(\underline{V}_I^{\otimes r},\underline{V}_J^{\otimes r})$, the fact that $E_{IJ}$ is a $G$-module homomorphism shows that $E_{IJ}\circ f \in \text{End}_G(\underline{V}_I^{\otimes r})$. Applying (\ref{G inv II}), we have $E_{IJ}\circ f \in \mathcal{B}_s^{(I)}(\varepsilon n)$, which yields that $f \in E_{JI}\mathcal{B}_s^{(I)}(\varepsilon n)=\mathcal{B}_{IJ}$. Hence $\text{Hom}_G(\underline{V}_I^{\otimes r},\underline{V}_J^{\otimes r})\subseteq\mathcal{B}_{IJ}$. On the other hand, $G$-module homomorphism $E_{IJ}$ and $\mathcal{B}_s^{(I)}(\varepsilon n)=\text{End}_G(\underline{V}_I^{\otimes r})$ conclude that $\mathcal{B}_{IJ}=E_{JI}\mathcal{B}_s^{(I)}(\varepsilon n)$ is a $G$-module homomorphism, i.e. $\mathcal{B}_{IJ} \subseteq \text{Hom}_G(\underline{V}_I^{\otimes r},\underline{V}_J^{\otimes r})$.
\end{proof}

\begin{remark}
	$\forall f \in \text{End}_G(V^{\otimes s})$, we have $T_{IJ}(f)=E_{JI}f^I=f^JE_{JI}$.
\end{remark}

\begin{corollary}\label{B_ll}
	For $0 \leq s \leq r$, $\text{End}_G(\underline{V}_s^{\otimes r})=\mathcal{B}(\varepsilon,n,r)_s$.
\end{corollary}
\begin{proof}
	It follows from (\ref{Hom decom 1}) that $\text{End}_G(\underline{V}_s^{\otimes r})=\underset{I,J \subseteq \underline{r} \atop \sharp I=\sharp J=s}{\bigoplus}\text{Hom}_G(\underline{V}_I^{\otimes r},\underline{V}_J^{\otimes r})$. Hence Lemma \ref{B_II 2} implies that $\text{End}_G(\underline{V}_s^{\otimes r})=\underset{I,J \subseteq \underline{r} \atop \sharp I=\sharp J=s}{\bigoplus}\mathcal{B}_{IJ}$. On the other hand, $\mathcal{B}(\varepsilon,n,r)_s=\underset{I,J \subseteq \underline{r} \atop \sharp I=\sharp J=s}{\bigoplus}\mathcal{B}_{IJ}$ by (\ref{B_l 1}).
\end{proof}

\begin{remark}\label{type A dual}
	If $G=\GL(V)$, we can similarly prove $\text{End}_G(\underline{V}_s^{\otimes r})=D(n,r)_s$ with $0 \leq s \leq r$.
\end{remark}

\textbf{Case 3. $s+t$ is even with $s \neq t$.}

Note that $s-t$ is even. We give some definitions which will be used in the following text.

For $K \subseteq L \subseteq \underline{r}$ with $\sharp L-\sharp K$ is even, we define 
\begin{align*}
P(K,L):=\{\{\{i_1,j_1\},\cdots,\{i_t,j_t\}\} \mid L-K=\{i_1,j_1,\cdots,i_t,j_t\}\}
\end{align*}
In particular, if $K=L$, $P(K,L)$ is an empty set. 

Define $\mathcal{B}_{IJ}$ as follows: If $s<t$,
\begin{align}\label{s<t}
\mathcal{B}_{IJ}:=\sum_{I \subseteq J' \atop \sharp J'=t}\sum_{s \in \frak S_t}\sum_{U \in P(I,J')}\sum_{z \in Z_s}\mathbb{C}E_{JJ'}(\Psi_t(s))^{J'}\mathcal{D}_U^{IJ'}(\tau_z)^I,
\end{align}
where $(\Psi_t(s))^{J'}$ is defined by (\ref{def of action}) and (\ref{Sl}); $Z_s$ is the set consists of all the normalized Brauer diagram with $2s$ dots; $(\tau_z)^I$ is defined by Lemma \ref{Brauer cal} and (\ref{def of action}). 

while $s>t$, 
\begin{align}\label{s>t}
\mathcal{B}_{IJ}:=\sum_{J' \subseteq I \atop \sharp J'=t}\sum_{s \in \frak S_t}\sum_{z \in Z_t}\sum_{U \in P(J',I)}\mathbb{C}E_{JJ'}(\Psi_t(s))^{J'}(\tau_z)^{J'}\mathcal{C}_U^{J'I}.
\end{align}

If $z$ is the graph with each dot in the top row connected with the dot below it, then $\tau_z=\text{id}$.

\begin{proposition}\label{B_{IJ} 2}
	If $I,J \subseteq \underline{r}$ satisfying $\sharp I-\sharp J$ is nonzero and even, then 
	\begin{align*}
	\text{Hom}_G(\underline{V}_I^{\otimes r},\underline{V}_J^{\otimes r})=\mathcal{B}_{IJ}.
	\end{align*}
\end{proposition}
\begin{proof}
	Let $\sharp I=s$ and $\sharp J=t$, we just assume that $s<t$. For the case $s>t$, the proof is similar. Assume that $f \in \text{Hom}_G(\underline{V}_I^{\otimes r},\underline{V}_J^{\otimes r})$, there exists $J' \subseteq \underline{r}$ such that $I \subseteq J'$ with $\sharp J'=t$. Take $\Pi \in P(I,J')$, now we consider $g=E_{J'J}\circ f \circ \mathcal{C}_\Pi^{IJ'} \in \text{End}_G(\underline{V}_{J'}^{\otimes r})$. By (\ref{to Brauer case 2}), we have $T^{-1}_{J'J'}(g) \in \text{End}_G(V^{\otimes t})=\mathcal{B}_t(\varepsilon n)$. Due to Lemma \ref{Brauer span}, $T^{-1}_{J'J'}(g)=\sum_{s \in \frak S_t}\sum_{z \in Z_t} a_{s,z}\Psi_t(s)\tau_z$ with $a_{s,z} \in \mathbb{C}$. It is clear that $g=(T^{-1}_{J'J'}(g))^{J'}$.
	
	Given $z=\{i_1,j_1\},\cdots,\{i_p,j_p\} \in Z_t$ with $p \leq [\frac{t}{2}]$ and let $I=\{\alpha_1,\cdots,\alpha_s\}$. Now we define $\text{Set}(z):=\{\alpha_{i_1},\alpha_{j_1},\cdots,\alpha_{i_p},\alpha_{j_p}\}$. Since $(T^{-1}_{J'J'}(g))^{J'}=(E_{J'J}\circ f) \circ \mathcal{C}_\Pi^{IJ'}$, the fact that $E_{J'J} \circ f$ is a linear map implies that $a_{s,z}=0$ for $\text{Set}(z) \neq \Pi$. So it is not hard to see 
	\begin{align}\label{Ef}
	E_{J'J}\circ f \in \sum_{s \in \frak S_t}\sum_{U \in P(I,J')}\sum_{z \in Z_s}\mathbb{C}(\Psi_t(s))^{J'}\mathcal{D}_U^{IJ'}(\tau_z)^I,
	\end{align}
which implies that $f \in \sum_{I \subseteq J' \atop \sharp J'=t}\sum_{s \in \frak S_t}\sum_{U \in P(I,J')}\sum_{z \in Z_s}\mathbb{C}E_{JJ'}(\Psi_t(s))^{J'}\mathcal{D}_U^{IJ'}(\tau_z)^I$. Hence we have $\text{Hom}_G(\underline{V}_I^{\otimes r},\underline{V}_J^{\otimes r}) \subseteq\mathcal{B}_{IJ}$. On the other hand, it is clear that $\mathcal{B}_{IJ}$ commutes with the action of $G$, i.e. $\mathcal{B}_{IJ} \subseteq \text{Hom}_G(\underline{V}_I^{\otimes r},\underline{V}_J^{\otimes r})$.
\end{proof}

\textbf{Describe $\text{End}_G(\underline{V}^{\otimes r})$.}

For $0 \leq s,t \leq r$ with $s+t$ is even, we define
\begin{align}\label{Bst def}
\mathcal{B}_{st}:=\bigoplus_{I \subseteq \underline{r} \atop \sharp I=s}\bigoplus_{J \subseteq \underline{r} \atop \sharp J=t}\mathcal{B}_{IJ}.
\end{align}
In particular, $\mathcal{B}_{ss}=\underset{I,J \subseteq \underline{r} \atop \sharp I=\sharp J=s}{\bigoplus}\mathcal{B}_{IJ}$. Hence $\mathcal{B}_{ss}=\mathcal{B}(\varepsilon,n,r)_s$. Then $\text{End}_G(\underline{V}_s^{\otimes r})=\mathcal{B}_{ss}=\underset{I,J \subseteq \underline{r} \atop \sharp I=\sharp J=s}{\bigoplus}\mathcal{B}_{IJ}$ (By Corollary \ref{B_ll}).

By (\ref{B_{IJ} odd}), Lemma \ref{B_II 2} and Proposition \ref{B_{IJ} 2}, the following corollary holds.

\begin{corollary}\label{B_{IJ} all}
	Let $G$ be $\text{O}(V)$ or $\text{Sp}(V)$ and assume that $I,J \subseteq \underline{r}$. Then
	\begin{align*}
	\text{Hom}_G(\underline{V}_I^{\otimes r},\underline{V}_J^{\otimes r})=\begin{cases}
	0, &  \text{if } \sharp I + \sharp J \text{ is odd;}\\
	\mathcal{B}_{IJ}, & \text{if } \sharp I + \sharp J \text{ is even.}
	\end{cases}
	\end{align*}
\end{corollary}

\begin{corollary}\label{Bst all}
	Let $G$ be $\text{O}(V)$ or $\text{Sp}(V)$ and assume that $0 \leq s,t \leq r$. Then
	\begin{align*}
	\text{Hom}_G(\underline{V}_s^{\otimes r},\underline{V}_t^{\otimes r})=\begin{cases}
	0, &  \text{if } s+t \text{ is odd;}\\
	\mathcal{B}_{st}, & \text{if } s+t \text{ is even.}
	\end{cases}
	\end{align*}
	In particular, $\text{End}_G(\underline{V}_s^{\otimes r})=\mathcal{B}_{ss}=\mathcal{B}(\varepsilon,n,r)_s$.
\end{corollary}
\begin{proof}
	If $s+t$ is even, $\text{Hom}_G(\underline{V}_s^{\otimes r},\underline{V}_t^{\otimes r})=0$ follows from (\ref{Bst odd}). 
	
	Now assume that $s+t$ is odd. By Corollary \ref{B_{IJ} all}, (\ref{Hom decom 1}) and (\ref{Bst def}), we have
	\begin{align*}
	\text{Hom}_G(\underline{V}_s^{\otimes r},\underline{V}_t^{\otimes r})=\bigoplus_{I \subseteq \underline{r} \atop \sharp I=s}\bigoplus_{J \subseteq \underline{r} \atop \sharp J=t}\text{Hom}_G(\underline{V}_I^{\otimes r},\underline{V}_J^{\otimes r})=\bigoplus_{I \subseteq \underline{r} \atop \sharp I=s}\bigoplus_{J \subseteq \underline{r} \atop \sharp J=t}\mathcal{B}_{IJ}=\mathcal{B}_{st}.
	\end{align*} In particular, take $s=t$, then $\text{End}_G(\underline{V}_s^{\otimes r})=\mathcal{B}_{ss}=\mathcal{B}(\varepsilon,n,r)_s$.
\end{proof}

\begin{thm} (Restricted duality)\label{restricted}
	Let $G$ be $\text{O}(V)$ or $\text{Sp}(V)$. Then 
	\begin{align*}
	\text{End}_G(\underline{V}^{\otimes r})=\bigoplus_{0 \leq s,t \leq r \atop s+t \text{ is even}}\mathcal{B}_{st}=\left(\bigoplus_{s-t \neq 0 \text{ is even}}\mathcal{B}_{st}\right) \bigoplus \mathcal{B}(\varepsilon,n,r).
	\end{align*}
\end{thm}

\begin{proof}
	The first equation follows from Corollary \ref{Bst all} and (\ref{End decom 1}). The second one is due to $\mathcal{B}_{ll}=\mathcal{B}(\varepsilon,n,r)_l$ ($0 \leq l \leq r$) and Lemma \ref{B decom}.
\end{proof}

\subsection{Levi and Parabolic dualities.}\label{LPTIs}
Keep the notations as before. In this section, we assume that $G$ is a closed subgroup of $\GL(V)$. Note that $G \times G_m$ and $\underline{G} \rtimes \textbf{G}_m$ are Levi group and parabolic group respectively. It is clear that (\ref{End decom 1}), (\ref{Hom decom 1}), (\ref{End decom 2}) and (\ref{Hom decom 2}) are valid when $G$ is a closed subgroup of $G$. 
Now we begin to consider $\text{Hom}_{G \times \textbf{G}_m}(\underline{V}_s^{\otimes r}, \underline{V}_t^{\otimes r})$.

Let $c \in \textbf{G}_m \subseteq G \times \textbf{G}_m$ and $v_1,\cdots,v_l \in V$ ($0 \leq l \leq r$), then $v_1\otimes \cdots \otimes v_l \otimes \eta^{(r-l)} \in \underline{V}_{\underline{l}}^{\otimes r}$. By the action of $\GL(\underline{V})$ on $\underline{V}$, we have $c.v_i=v_i$ ($1\leq i \leq l$) and $c.\eta=c\eta$. Thus
\begin{align}\label{c action}
\Phi(c)(v_1\otimes \cdots \otimes v_l \otimes \eta^{(r-l)})=v_1\otimes \cdots \otimes v_l \otimes (c\eta)^{(r-l)}=c^{r-l}v_1\otimes \cdots \otimes v_l \otimes \eta^{(r-l)}.
\end{align}

$\forall I \subseteq \underline{r}$, (\ref{c action}) implies that
\begin{align}\label{c action 2}
\Phi|_I(c)=\Phi(c)|_{\underline{V}_I^{\otimes r}}=c^{r-\sharp I}\text{id}_{\underline{V}_I^{\otimes r}}.
\end{align}

Hence for $0 \leq l \leq r$, we have
\begin{align}\label{c action 3}
\Phi|_l(c)=\Phi(c)|_{\underline{V}_l^{\otimes r}}=c^{r-l}\text{id}_{\underline{V}_l^{\otimes r}}.
\end{align}

Assume $I \subseteq \underline{r}$ with $\sharp I=l$, and $c \in \textbf{G}_m$. Keep in mind that $T_{II}$ gives an isomorphism of vector spaces between $\text{End}_G(V^{\otimes l})$ and $\text{End}_G(\underline{V}_I^{\otimes r})$ (see (\ref{to Brauer case 2})). Clearly, $\Phi|_I(c)=c^{r-l}\text{id}_{\underline{V}_I^{\otimes r}} \in \text{End}_G(V^{\otimes l})$. It is not hard to see $T_{II}^{-1}(\Phi|_I(c))=c^{r-l}\text{id}_{V^{\otimes l}}$. This means that $\Phi|_I(c)$ can be regarded as $c^{r-l}\text{id}_{V^{\otimes l}} \in \text{End}_G(V^{\otimes l})$. Equations (\ref{c action 2}) and (\ref{c action 3}) shows that $\underline{V}^{\otimes r}$, $\underline{V}_l^{\otimes r}$ and $\underline{V}_I^{\otimes r}$ are $G \times \textbf{G}_m$-modules.

\begin{lemma}\label{Levi 1}
	Assume $I,J \subseteq \underline{r}$ satisfying $\sharp I \neq \sharp J$, then $\text{Hom}_{G \times \textbf{G}_m}(\underline{V}_I^{\otimes r}, \underline{V}_J^{\otimes r})=0$.
\end{lemma}
\begin{proof}
	Set $s=\sharp I$ and $t=\sharp J$, then $0 \leq s \neq t \leq r$. For $T \in \text{Hom}_{G \times \textbf{G}_m}(\underline{V}_I^{\otimes r}, \underline{V}_J^{\otimes r})$, we have $T$ commutes with the action of $\textbf{G}_m$, i.e. $\forall c \in \textbf{G}_m$, $T\circ\Phi|_I(c)=\Phi|_J(c)\circ T$. Then (\ref{c action 2}) shows that $T\circ c^{r-s}\text{id}_{\underline{V}_I^{\otimes r}}=c^{r-t}\text{id}_{\underline{V}_J^{\otimes r}}\circ T$, i.e. $c^{r-s}T=c^{r-t}T$. Hence $T=0$. 
\end{proof}

By the proof of Lemma \ref{Levi 1}, the following Lemma holds.

\begin{lemma}\label{Levi 5}
	Assume that $I,J \subseteq \underline{r}$ with $\sharp I=\sharp J$. Then 
	\begin{align*}
	\text{Hom}_{G \times \textbf{G}_m}(\underline{V}_I^{\otimes r},\underline{V}_J^{\otimes r})=\text{Hom}_{G}(\underline{V}_I^{\otimes r},\underline{V}_J^{\otimes r}).
	\end{align*}
\end{lemma}

\begin{lemma}\label{Bst Levi}
	For $0 \leq s,t \leq r$, we have 
	\begin{align*}
		\text{Hom}_{G \times \textbf{G}_m}(\underline{V}_s^{\otimes r},\underline{V}_t^{\otimes r})=\begin{cases}
		0, & \text{if }s \neq t;\\
		\text{Hom}_{G}(\underline{V}_s^{\otimes r},\underline{V}_t^{\otimes r}), & \text{if }s=t.
		\end{cases}
	\end{align*}
\end{lemma}
\begin{proof}
	The result follows from Lemma \ref{Levi 1}, Lemma \ref{Levi 5}, (\ref{Hom decom 1}) and (\ref{Hom decom 2}). 
\end{proof}

\begin{remark}\label{Bst Levi 2}
	Take $s=t$ in Lemma \ref{Bst Levi}, it shows: for $0 \leq l \leq r$, then 
	\begin{align*}
	\text{End}_{G \times \textbf{G}_m}(\underline{V}_l^{\otimes r})=\text{End}_{G}(\underline{V}_l^{\otimes r}).
	\end{align*}
	Moreover, by Remark \ref{type A dual} and Corollary \ref{B_ll}, we have
	\begin{align}\label{ABCDl}
	\text{End}_{G \times \textbf{G}_m}(\underline{V}_l^{\otimes r})=\text{End}_{G}(\underline{V}_l^{\otimes r})=\begin{cases}
	D(n,r)_l, & \text{if }G=\GL(V);\\
	\mathcal{B}(\varepsilon,n,r)_l, & \text{if }G=\text{O}(V) \text{ or }G=\text{Sp}(V).
	\end{cases}
	\end{align}
\end{remark}

\begin{corollary}\label{inclusion1}
There is an isomorphism of algebras
\begin{align*}
\text{End}_{G \times \textbf{G}_m}(\underline{V}^{\otimes r}) \cong \bigoplus_{l=0}^r\text{End}_{G \times \textbf{G}_m}(\underline{V}_l^{\otimes r})=\bigoplus_{l=0}^r\text{End}_{G}(\underline{V}_l^{\otimes r}).
\end{align*}
\end{corollary}
\begin{proof}
	 This statement follows from \eqref{End decom 2} and Lemma \ref{Bst Levi}. 
\end{proof}

Now it is a position to give the duality related to Levi subgroup $G \times \textbf{G}_m$.

\begin{thm} (Levi duality)\label{AB Dual}
	Keep the notations as above. Then
	\begin{align*}
	\text{End}_{G \times \textbf{G}_m}(\underline{V}^{\otimes r})=\begin{cases}
	D(n,r), & \text{if }G=\GL(V);\\
	\mathcal{B}(\varepsilon,n,r), & \text{if }G=\text{Sp}(V)\text{  or }G=\text{O}(V).
	\end{cases}
	\end{align*}
\end{thm}

\begin{proof}
	The result follows from Corollary \ref{inclusion1}, Proposition \ref{B decom}, \cite[Lemma 3.5(1)]{BYY} and (\ref{ABCDl}).
\end{proof}

For $M \subseteq \text{End}_\mathbb{C}(\underline{V}^{\otimes r})$, we define $M^V:=\{\phi \in M \mid \Phi(e^v) \circ \phi=\phi \circ \Phi(e^v), \forall v \in V\}$. If $M$ is a subalgebra of $\text{End}_\mathbb{C}(\underline{V}^{\otimes r})$, then $M^{V}$ is a subalgebra of $\text{End}_\mathbb{C}(\underline{V}^{\otimes r})$ as well. The duality related to parabolic subgroup $\underline{G} \rtimes \textbf{G}_m$ can be easily obtained.

\begin{corollary} (Parabolic duality)\label{inclusion3}
	Let $G$ be a closed subgroup of $\GL(V)$, then 
	\begin{align*}
	\text{End}_{\underline{G} \rtimes \textbf{G}_m}(\underline{V}^{\otimes r})=\left( \text{End}_{G \times \textbf{G}_m}(\underline{V}^{\otimes r}) \right)^V.
	\end{align*}
	In particular, Theorem \ref{AB Dual} implies that
	\begin{align}\label{P Dual}
	\text{End}_{\underline{G} \rtimes \textbf{G}_m}(\underline{V}^{\otimes r})=\begin{cases}
	D(n,r)^V, & \text{if }G=\GL(V);\\
	\mathcal{B}(\varepsilon,n,r)^V, & \text{if }G=\text{Sp}(V)\text{  or }G=\text{O}(V).
	\end{cases}
	\end{align}
\end{corollary}

\begin{proof}
	$\forall f \in \text{End}_{\underline{G} \rtimes \textbf{G}_m}(\underline{V}^{\otimes r})$, then $f \circ \Phi(e^v)=\Phi(e^v) \circ f$ for all $v \in V$. Since $G \times \textbf{G}_m$ is a subgroup of $\underline{G} \rtimes \textbf{G}_m$, we have $f \in \text{End}_{G \times \textbf{G}_m}(\underline{V}^{\otimes r})$. Thus $f \in \left( \text{End}_{G \times \textbf{G}_m}(\underline{V}^{\otimes r}) \right)^V$. Namely, $\text{End}_{\underline{G} \rtimes \textbf{G}_m}(\underline{V}^{\otimes r}) \subseteq \left( \text{End}_{G \times \textbf{G}_m}(\underline{V}^{\otimes r}) \right)^V$.

	Conversly, assume $f \in \left( \text{End}_{G \times \textbf{G}_m}(\underline{V}^{\otimes r}) \right)^V$. It is clear that $f \in \text{End}_{G \times \textbf{G}_m}(\underline{V}^{\otimes r})$ as well as $f \circ \Phi(e^v)=\Phi(e^v) \circ f$ for all $v \in V$. For all $ \underline{g}=(g,v) \in \underline{G}$ with $g \in G$ and $v \in V$, we see that $\underline{g}=(g,0)\cdot e^{g^{-1}(v)}=e^{c^{-1}v} \cdot (g,0)$. Thus $f \in \text{End}_{\underline{G} \rtimes \textbf{G}_m}(\underline{V}^{\otimes r})$). 
\end{proof}

\begin{remark}
	(\ref{P Dual}) shows that $\text{End}_{\underline{\GL(V)} \rtimes \textbf{G}_m}(\underline{V}^{\otimes r})=D(n,r)^V$, which is just \cite[Theorem 5.3(1)]{BYY} and \cite[Theorem 1.4]{LB}. 
	
	It is clear that $\Phi(\underline{G} \rtimes \textbf{G}_m)$ commutes with $\Psi(\frak S_r)$, i.e. $\mathbb{C}\Psi(\frak S_r) \subseteq \text{End}_{\underline{G} \rtimes \textbf{G}_m}(\underline{V}^{\otimes r})$.
\end{remark}

\section{Tensor invariants of parabolic groups.}\label{3}
Keep the notations as before. For $w \in V$ and $\textbf{v} \in \underline{V}^{\otimes r}$, we define \begin{align}\label{D}
	D_{\textbf{v}}^w:=\Phi(e^w)(\textbf{v})-\textbf{v}.
\end{align} 

Assume that $\textbf{v}'=w_1 \otimes \cdots \otimes w_t \otimes \eta^{\otimes (r-t)}$ with $0 \leq t \leq r$ and $w_1,\cdots,w_t \in V$. Since $e^w(\eta)=w+\eta$ and $e^w|_{V}=\text{id}_V$, then (\ref{tensor action}) shows that $D_{\textbf{v}'}^w=w_1 \otimes \cdots \otimes w_t \otimes (w+\eta)^{\otimes (r-t)}-w_1 \otimes \cdots \otimes w_t \otimes \eta^{\otimes (r-t)}=w_1 \otimes \cdots \otimes w_t \otimes (w+\eta)^{\otimes (r-t)}-\textbf{v}'$.

Recall that $\rho_I: \mathcal{B}_r(\varepsilon n) \rightarrow \text{End}_\mathbb{C}(\underline{V}^{\otimes r})$ is defined in (\ref{tau action 4}), which is a linear map. For $x \in \mathcal{B}_r(\varepsilon n)$, we define $\rho(x):=\sum_{I \subseteq \underline{r}}\rho_I(x)$. It is clear that $\rho: \mathcal{B}_r(\varepsilon n) \rightarrow \text{End}_\mathbb{C}(\underline{V}^{\otimes r})$ is also a linear map. By (\ref{special}), we have \begin{align}\label{rho action}
\rho(\Psi|_r(s))=\Psi(s).
\end{align}

\begin{lemma}\label{key lemma}
	Assume $G=\text{O}(V)$ with $\dim V=n \geq 2r$. Suppose that $\delta \in \mathcal{B}_r(n)$ and $J \subsetneq \underline{r}$. If $\forall w \in V$ and $\forall \textbf{v} \in \underline{V}_J^{\otimes r}$, we always have $\rho(\delta)(D_\textbf{v}^w)=0$, then $\rho(\delta)(\underline{V}_J^{\otimes r})=0$, i.e. $\rho_J(\delta)=0$.
\end{lemma}

\begin{proof}
	Denote $l=:\sharp J$. Suppose $G$ leaves invariant a nondegenerate bilinear form $\omega$ on $V$. For a basis $\{f_p \mid 1 \leq p \leq n\}$ for $V$ , denote by $\{f^p \mid 1 \leq p \leq n\}$ the dual basis for $V$ relative to $\omega$. Namely, $w(f_p,f^q)=\delta_{pq}$ with $1 \leq p,q \leq n$.
	
	By Lemma \ref{Brauer basis}, we assume that $\delta=\sum_{s \in \frak S_r} \sum_{z \in Z_r} c_{s,z}\Psi|_r(s)\tau_z$ with $c_{s,z} \in \mathbb{C}$. Let $\delta_1=\sum_{s \in \frak S_r} \sum_{z \in Z_r \atop e(z) \subseteq J} c_{s,z}\Psi|_r(s)\tau_z$ and $\delta_2=\sum_{s \in \frak S_r} \sum_{z \in Z_r \atop e(z) \nsubseteq J} c_{s,z}\Psi|_r(s)\tau_z$, then $\delta=\delta_1 + \delta_2$ (Note that $e(z)$ is defined in \S\ref{algebra B}). Hence we have $\rho(\delta)=\rho(\delta_1) + \rho(\delta_2)$ by the definition of $\rho$. 
	
	Given $\textbf{v} \in \underline{V}_J^{\otimes r}$, suppose that $\rho(\delta_2)(D_\textbf{v}^w)\neq 0$. If $u_1 \otimes \cdots \otimes u_r$ is a summand of $\rho(\delta_2)(D_\textbf{v}^w)$, then $\sharp \{u_i \mid u_i=w \text{ or } u_i=\eta, \text{ with }1 \leq i \leq r\}<r-l$. While if $u'_1 \otimes \cdots \otimes u'_r$ is a summand of $\rho(\delta_1)(D_\textbf{v}^w)$, then $\sharp \{u'_i \mid u'_i=w \text{ or } u'_i=\eta, \text{ with }1 \leq i \leq r\}=r-l$. 
	
	We can choose a proper basis $\{f_p \mid 1 \leq p \leq n\}$ for $V$, such that $f_p=f^p$ for all $1 \leq p \leq n$. Without any loss of generality, we assume that $J=\underline{l}$ with $0 \leq l \leq r$. Now we take $w_1,\cdots,w_l \in \{f_1,\cdots,f_n\}$. Since $n \geq 2r$ and $r>l$, then $n>l+1$, thus there exists $1 \leq t \leq n$, such that $f_t \neq w_i$ for all $i \in \underline{l}$. Hence $f_t \notin \sum_{i=0}^{l}\mathbb{C}w_i \oplus \mathbb{C}(f_1+\cdots+f_n)$. Take $\textbf{v}'=w_1 \otimes \cdots \otimes w_l \otimes \eta^{(n-l)} \in \underline{V}_{\underline{l}}^{\otimes r}$, then we have $\rho(\delta)(D_{\textbf{v}'}^{f_t})=0$. If $\rho(\delta_2)(D_{\textbf{v}'}^{f_t})\neq 0$, then each summand of $\rho(\delta_2)(D_{\textbf{v}'}^{f_t})=\rho(\delta_2)(w_1 \otimes \cdots \otimes w_l \otimes {f_t}^{(n-l)}+\cdots)$ can not be cancelled out in $\rho(\delta)(D_{\textbf{v}'}^{f_t})$. This contradicts with the fact $\rho(\delta)(D_{\textbf{v}'}^{f_t})=0$. So we have $\rho(\delta_2)(D_{\textbf{v}'}^{f_t})=0$. Hence $\rho(\delta_1)(D_{\textbf{v}'}^{f_t})=0$. The fact $\rho(\delta_1)(D_{\textbf{v}'}^{f_t})=0$ implies that all the summand of $\rho(\delta_1)(D_{\textbf{v}'}^{f_t})=\rho(\delta_1)(w_1 \otimes \cdots \otimes w_l \otimes {f_t}^{(n-l)}+\cdots)$ cancel each other out. So $\rho(\delta_1)(w_1 \otimes \cdots \otimes w_l \otimes {f_t}^{(n-l)}+\cdots)=0$ is also valid when $f_t$ is replaced by $\eta$, i.e. $\rho(\delta_1)(w_1 \otimes \cdots \otimes w_l \otimes \eta^{(n-l)}+\cdots)=0$. Hence $\rho(\delta_1)(w_1 \otimes \cdots \otimes w_l \otimes \eta^{(n-l)})=0$. Since $\textbf{v}'=w_1 \otimes \cdots \otimes w_l \otimes \eta^{(n-l)}$ runs through a basis of $\underline{V}_J^{\otimes r}$, then $\rho(\delta_1)(\underline{V}_J^{\otimes r})=0$. It is clear that $\rho(\delta_2)(\underline{V}_J^{\otimes r})=0$ by definition, then $\rho(\delta)(\underline{V}_J^{\otimes r})=0$.
\end{proof}

\begin{lemma}\label{key lemma 2}
	Assume $G=\text{Sp}(V)$ with $\dim V=n>2r$. Suppose that $\delta \in \mathcal{B}_r(-n)$ and $J \subsetneq \underline{r}$. If $\forall w \in V$ and $\forall \textbf{v} \in \underline{V}_J^{\otimes r}$, we always have $\rho(\delta)(D_\textbf{v}^w)=0$, then $\rho(\delta)(\underline{V}_J^{\otimes r})=0$, i.e. $\rho_J(\delta)=0$.
\end{lemma}

\begin{proof}
	The proof is similar to Lemma \ref{key lemma}.
\end{proof}

\begin{thm}\label{inv thm}
	Assume $G=\text{Sp}(V)$ with $\dim V>2r$, or $G=\text{O}(V)$ with $\dim V \geq 2r$, then we have $\mathcal{B}(\varepsilon,n,r)^V=\{\rho(f) \mid f \in \mathcal{B}_r(\varepsilon n)\}$, i.e. $\text{End}_{\underline{G} \rtimes \textbf{G}_m}(\underline{V}^{\otimes r})=\{\rho(f) \mid f \in \mathcal{B}_r(\varepsilon n)\}$.
\end{thm}

\begin{proof}
	Suppose $\dim V=n$, and assume that $\sigma \in \mathcal{B}(\varepsilon,n,r)^V$. Equation (\ref{for final use 2}) implies that there exist $f_I \in \mathcal{B}_r(\varepsilon n)$ with all $I \subseteq \underline{r}$, such that $\sigma=\sum_{I \subseteq \underline{r}}\rho_I(f_I)$. Now we use induction to prove.
	
	If $\sharp I=r$, then it must have $I=\underline{r}$. So we assume that there exist $f \in \mathcal{B}_r(\varepsilon n)$ with $0 \leq k \leq r-1$, such that $\forall I \subseteq \underline{r}$ with $\sharp I \geq k+1$ we have $\rho_I(f_I)=\rho_I(f)$. Then $\sigma=\sum\limits_{I \subseteq \underline{r}, ~ \sharp I \leq k} \rho_I(f_I)+\sum\limits_{I \subseteq \underline{r}, ~ \sharp I >k} \rho_I(f)$. It is easy to see $\forall w \in V,e^w(\eta)=w+\eta$ and $e^w|_{V}=\text{id}_V$. For all $v_1,\cdots,v_k \in V$ and $\forall w \in V$, we have $v_1 \otimes \cdots \otimes v_k \otimes (w+\eta)^{\otimes (r-k)}=v_1 \otimes \cdots \otimes v_k \otimes \eta^{\otimes (r-k)}+D_{\textbf{v}}^w$, where $\textbf{v}:=v_1 \otimes \cdots \otimes v_k \otimes \eta^{\otimes (r-k)}$. By (\ref{tensor action}), we have
	\begin{align*}
	(\sigma \circ \Phi(e^w))(v_1 \otimes \cdots \otimes v_k \otimes \eta^{\otimes (r-k)})
	=\sigma(v_1 \otimes \cdots \otimes v_k \otimes (w+\eta)^{\otimes (r-k)})
	\end{align*}
	\begin{align*}
	=(\sum\limits_{I \subseteq \underline{r}, ~ \sharp I \leq k} \rho_I(f_I)+\sum\limits_{I \subseteq \underline{r}, ~ \sharp I >k} \rho_I(f))(v_1 \otimes \cdots \otimes v_k \otimes (w+\eta)^{\otimes (r-k)})
	\end{align*}
	\begin{align*}
	=(\rho_{\underline{k}}(f_{\underline{k}})+\sum\limits_{I \subseteq \underline{r}, ~ \sharp I >k} \rho_I(f))(v_1 \otimes \cdots \otimes v_k \otimes (w+\eta)^{\otimes (r-k)})
	\end{align*}
	\begin{align*}
	=\rho_{\underline{k}}(f_{\underline{k}})(v_1 \otimes \cdots \otimes v_k \otimes \eta^{\otimes (r-k)})+\rho(f)(D_{\textbf{v}}^w)
	\end{align*}
	\begin{align}\label{eq3}
	=\rho(f_{\underline{k}})(v_1 \otimes \cdots \otimes v_k \otimes \eta^{\otimes (r-k)})+\rho(f)(D_{\textbf{v}}^w).
	\end{align}
	
	By Lemma \ref{Brauer basis}, we suppose that $f_{\underline{k}}=\sum_{s \in \frak S_r} \sum_{z \in Z_r} c_{s,z}\Psi|_r(s)\tau_z$
with $c_{s,z} \in \mathbb{C}$. Let $g=\sum_{s \in \frak S_r} \sum_{z \in Z_r \atop e(z) \subseteq \underline{k}} c_{s,z}\Psi|_r(s)\tau_z$ and $h=\sum_{s \in \frak S_r} \sum_{z \in Z_r \atop e(z) \nsubseteq \underline{k}} c_{s,z}\Psi|_r(s)\tau_z$, then $f_{\underline{k}}=g+h$ (Note that $e(z)$ is defined in \S\ref{algebra B}). Hence we have $\rho(f_{\underline{k}})=\rho(g) + \rho(h)$ by the definition of $\rho$. Obviously, $\rho(h)(v_1 \otimes \cdots \otimes v_k \otimes \eta^{\otimes (r-k)})=\rho_{\underline{k}}(h)(v_1 \otimes \cdots \otimes v_k \otimes \eta^{\otimes (r-k)})=0$. Hence \begin{align}\label{eq3-1}
\rho(f_{\underline{k}})(v_1 \otimes \cdots \otimes v_k \otimes \eta^{\otimes (r-k)})=\rho(g)(v_1 \otimes \cdots \otimes v_k \otimes \eta^{\otimes (r-k)}).
\end{align}	Hence, we have \begin{align}\label{eq3-1.5}
\rho_{\underline{k}}(f_{\underline{k}})=\rho_{\underline{k}}(g).
\end{align}  

Equation (\ref{eq3}) and (\ref{eq3-1}) implies that 
\begin{align}\label{eq3-2}
(\sigma \circ \Phi(e^w))(v_1 \otimes \cdots \otimes v_k \otimes \eta^{\otimes (r-k)})=\rho(g)(v_1 \otimes \cdots \otimes v_k \otimes \eta^{\otimes (r-k)})+\rho(f)(D_{\textbf{v}}^w).
\end{align}

On the other hand, 	
	\begin{align*}
	(\Phi(e^w) \circ \sigma)(v_1 \otimes \cdots \otimes v_k \otimes \eta^{\otimes (r-k)})
	=\Phi(e^w)(\rho_{\underline{k}}(f_{\underline{k}})(v_1 \otimes \cdots \otimes v_k \otimes \eta^{\otimes (r-k)}))
	\end{align*}
	\begin{align*}
	=\Phi(e^w)\left(\rho(f_{\underline{k}})(v_1 \otimes \cdots \otimes v_k \otimes \eta^{\otimes (r-k)})\right) \overset{(\ref{eq3-1})}{=} \Phi(e^w)\left(\rho(g)(v_1 \otimes \cdots \otimes v_k \otimes \eta^{\otimes (r-k)})\right)
	\end{align*}
	\begin{align*}
	=\rho(g)\left(\Phi(e^w)(v_1 \otimes \cdots \otimes v_k \otimes \eta^{\otimes (r-k)})\right) \text{ }(\text{By Proposition \ref{B dual}})
	\end{align*}
	\begin{align*}
	=\rho(g)(v_1 \otimes \cdots \otimes v_k \otimes (w+\eta)^{\otimes (r-k)}) 
	\end{align*}
	\begin{align}\label{eq4}
	=\rho(g)(v_1 \otimes \cdots \otimes v_k \otimes \eta^{\otimes (r-k)})+\rho(g)(D_{\textbf{v}}^w)
	\end{align}
	
	Since $\sigma \in \mathcal{B}(\varepsilon,n,r)^V$, i.e. $\Phi(e^w) \circ \sigma=\sigma \circ \Phi(e^w)$, ($\ref{eq3-2}$) and ($\ref{eq4}$) yields that $\forall w \in V$, $\rho(f)(D_{\textbf{v}}^w)=\rho(g)(D_{\textbf{v}}^w)$. So we have $\rho(f)(D_{\textbf{u}}^w)=\rho(g)(D_{\textbf{u}}^w)$ for all $\textbf{u} \in \underline{V}_{\underline{k}}^{\otimes r}$ and $\forall w \in W$, namely $\rho(f-g)(D_{\textbf{u}}^w)=0$ for all $\textbf{u} \in \underline{V}_{\underline{k}}^{\otimes r}$ and $\forall w \in W$. Since $k \leq r-1$, due to Lemma $\ref{key lemma}$ and Lemma $\ref{key lemma 2}$, we conclude that $\left(\rho(f)-\rho(g)\right)(\underline{V}_{\underline{k}}^{\otimes r})=0$. Namely, $\rho(f)|_{\underline{V}_{\underline{k}}^{\otimes r}}=\rho(g)|_{\underline{V}_{\underline{k}}^{\otimes r}}$, which implies that $\rho_{\underline{k}}(f)=\rho_{\underline{k}}(g)$. Due to (\ref{eq3-1.5}), we have $\rho_{\underline{k}}(f)=\rho_{\underline{k}}(f_{\underline{k}})$. Generally, we can similarly prove that $\rho_I(f)=\rho_I(f_I)$ for all $I \subseteq \underline{r}$ with $\sharp I=k$. By induction on $\sharp I$, we have $\rho_I(f)=\rho_I(f_I)$ for all $I \subseteq \underline{r}$. Thus $\sigma=\sum_{I \subseteq \underline{r}}\rho_I(f_I)=\sum_{I \subseteq \underline{r}}\rho_I(f)=\rho(f)$. So we prove $\mathcal{B}(\varepsilon,n,r)^V \subseteq \{\rho(f) \mid f \in \mathcal{B}_r(\varepsilon n)\}$.

	Conversely, for $f \in \mathcal{B}_r(\varepsilon n)$, we need to show that $\rho(f) \in \mathcal{B}(\varepsilon,n,r)^V$. Lemma \ref{pre lemma} implies that $\rho(f) \in \mathcal{B}(\varepsilon,n,r)$. For all $v_1,\cdots,v_k \in V$ and $\forall w \in V$, imitate the proof of (\ref{eq3-2}) and (\ref{eq4}), we can yield that $(\rho(f) \circ \Phi(e^w))(v_1 \otimes \cdots \otimes v_k \otimes \eta^{\otimes (r-k)})=(\Phi(e^w) \circ \rho(f))(v_1 \otimes \cdots \otimes v_k \otimes \eta^{\otimes (r-k)})$. Hence $\forall I \subseteq \underline{r}$ and $\textbf{v} \in \underline{V}_I^{\otimes r}$, we have $(\rho(f) \circ \Phi(e^w))(\textbf{v})=(\Phi(e^w) \circ \rho(f))(\textbf{v})$. This shows that $\rho(f) \circ \Phi(e^w)=\Phi(e^w) \circ \rho(f)$, which means that $\rho(f)  \in \mathcal{B}(\varepsilon,n,r)^V$. Namely, $\{\rho(f) \mid f \in \mathcal{B}_r(\varepsilon n)\} \subseteq \mathcal{B}(\varepsilon,n,r)^V$.	
\end{proof}

\section*{Comments}
This is a primitive version.

\section*{Acknowledgment}
The author thanks Bin Shu for his guidance and suggestions on this topic.

\end{document}